\documentclass[a4paper,12pt]{article}

\usepackage{amsmath,amssymb,amsthm,amscd,enumerate}
\usepackage{color}

\allowdisplaybreaks

\newcommand{\Fg}{\mathfrak{g}}
\newcommand{\Fh}{\mathfrak{h}}
\newcommand{\Fb}{\mathfrak{b}}

\newcommand{\Fa}{\mathfrak{a}}
\newcommand{\FS}{\mathfrak{S}}

\newcommand{\Fs}{\mathfrak{s}}
\newcommand{\Fu}{\mathfrak{u}}
\newcommand{\Fz}{\mathfrak{z}}
\newcommand{\Ft}{\mathfrak{t}}

\newcommand{\BC}{\mathbb{C}}
\newcommand{\BR}{\mathbb{R}}

\newcommand{\BZ}{\mathbb{Z}}
\newcommand{\BF}{\mathbb{F}}

\newcommand{\CC}{\mathcal{C}}

\newcommand{\Hom}{\mathop{\rm Hom}\nolimits}
\newcommand{\End}{\mathop{\rm End}\nolimits}
\newcommand{\Ind}{\mathop{\rm Ind}\nolimits}
\newcommand{\Aut}{\mathop{\rm Aut}\nolimits}

\newcommand{\Ni}{\mathop{\rm Ni}\nolimits}
\newcommand{\Span}{\mathop{\rm Span}\nolimits}
\newcommand{\rank}{\mathop{\rm rank}\nolimits}
\newcommand{\id}{\mathop{\rm id}\nolimits}
\newcommand{\ad}{\mathop{\rm ad}\nolimits}

\newcommand{\pair}[2]{\langle #1,\,#2 \rangle}

\newcommand{\ve}{\varepsilon}
\newcommand{\vp}{\varphi}

\newcommand{\ha}[1]{\widehat{#1}}
\newcommand{\ti}[1]{\widetilde{#1}}
\newcommand{\ol}[1]{\overline{#1}}

\newcommand{\kc}{\langle \kappa \rangle}
\newcommand{\ten}{\mathrm{X}}

\newcommand{\bzero}{{\bf 0}}
\newcommand{\bk}{{\bf k}}
\newcommand{\bh}{{\bf h}}

\newcommand{\bmu}{{\boldsymbol \mu}}

\newcommand{\bqed}{\quad \hbox{\rule[-0.5pt]{3pt}{8pt}}}

\makeatletter
\@addtoreset{equation}{subsection}

\renewcommand\section{\@startsection{section}{1}{0pt}
{-3.5ex plus -1ex minus -.2ex}{1.0ex plus .2ex}{\large\bf}}
\renewcommand\subsection{\@startsection{subsection}{1}{0pt}
{2.5ex plus 1ex minus .2ex}{-1em}{\bf}}
\makeatother

\theoremstyle{plain}
\newtheorem{thm}{Theorem}[subsection]
\newtheorem{lem}[thm]{Lemma}
\newtheorem{prop}[thm]{Proposition}

\newtheorem{claim}{Claim}[thm]

\newtheorem*{claim*}{Claim}

\theoremstyle{definition}

\theoremstyle{remark}
\newtheorem{rem}[thm]{Remark}
\newtheorem*{rem*}{Remark}

\begin{document}

\setlength{\baselineskip}{18.3pt}

\title{\Large\bf 
Application of a $\BZ_{3}$-orbifold construction \\[1.5mm]
to the lattice vertex operator algebras \\[1.5mm]
associated to Niemeier lattices
}
\author{
 Daisuke Sagaki%
\footnote{D.S. was partially supported by 
Grant-in-Aid for Young Scientists (B) No.\,23740003, Japan.} \\
 \small Institute of Mathematics, University of Tsukuba, \\
 \small Tennodai 1-1-1, Tsukuba, Ibaraki 305-8571, Japan \\
 \small (e-mail: {\tt sagaki@math.tsukuba.ac.jp}) \\[5mm]
 Hiroki Shimakura%
\footnote{H.S. was partially supported by 
Grant-in-Aid for Scientific Research (C) No.\,23540013, Japan.} \\
 \small Graduate School of Information Sciences, Tohoku University, \\
 \small Aramaki aza Aoba 6-3-09, Aoba-ku, Sendai 980-8579, Japan \\
 \small (e-mail: {\tt shimakura@m.tohoku.ac.jp})
}
\date{}
\maketitle

%
\begin{abstract} \setlength{\baselineskip}{16pt}
%
%
By applying Miyamoto's $\BZ_{3}$-orbifold construction 
to the lattice vertex operator algebras associated to 
Niemeier lattices and their automorphisms of order $3$, 
we construct holomorphic vertex operator algebras of central charge $24$ 
whose Lie algebras of the weight one spaces are of types $A_{2,3}^6$, 
$E_{6,3}G_{2,1}^{3}$, and $A_{5,3}D_{4,3}A_{1,1}^{3}$,
which correspond to No.\,6, No.\,17, and No.\,32 on Schellekens' list, 
respectively. 
\end{abstract}
%
%
\section{Introduction.}
\label{sec:intro}
%
%
The classification of holomorphic vertex operator algebras 
(VOAs for short) is a fundamental problem in the VOA theory.
By Zhu's theory (see \cite{Z}), their central charges are 
divisible by $8$. We know from \cite{DM04} that 
a holomorphic VOA of central charge $8$ or $16$ is 
isomorphic to a lattice VOA. Thus the next problem is 
to classify the holomorphic VOAs of central charge $24$. 
In \cite{Sch}, Schellekens gave a list of 
possible $71$ Lie algebra structures of 
the weight one spaces of holomorphic VOAs of central charge $24$ 
(it is mysterious that the number $71$ appears here; for, 
$71$ is the largest prime factor of the order of the Monster simple group). 
So the first step of the classification would be: 
\begin{quote}
For each Lie algebra on Schellekens' list, 
construct a holomorphic VOA of central charge $24$ 
whose Lie algebra of the weight one space is isomorphic to it. 
\end{quote}

It is well-known that the lattice VOA associated to 
a Niemeier lattice (i.e., unimodular, positive-definite, even lattice of rank $24$) 
is a holomorphic VOA of central charge $24$. 
Because there exist exactly $24$ Niemeier lattices 
(see, e.g., \cite[Chapter 16, Table 16.1]{CS}), 
we can obtain $24$ holomorphic VOAs of central charge $24$ 
in this manner. 
Also, we can obtain $15$ holomorphic VOAs of central charge $24$ 
by applying the $\BZ_2$-orbifold construction to the lattice VOAs associated to 
Niemeier lattices and the $(-1)$-isometry (see \cite{FLM,DGM}); 
the Moonshine VOA is included in these 15 holomorphic VOAs. 
In \cite{Lam,LS1,LS2}, 
Lam and the second named author proved 
that there exist exactly $56$ 
holomorphic framed VOAs of central charge $24$, 
including $39\,(=24+15)$ holomorphic VOAs mentioned above. 
According to Schellekens' list, there should 
exist at least $15\,(=71-56)$ non-framed holomorphic VOAs 
of central charge $24$, which should correspond to: 
\begin{equation*}
\begin{array}{c|c|c}
\text{ No. in \cite{Sch} }
& 
\text{ \begin{tabular}{cc} Dimension of \\ the weight one space \end{tabular} }
& 
\text{ \begin{tabular}{cc} Lie algebra structure of \\ 
       the weight one space \end{tabular}} \\ \hline\hline
3 & 36 & D_{4,12}A_{2,6} \\[1mm]
4 & 36 & C_{4,10} \\[1mm]
6 & 48 & A_{2,3}^{6} \\[1mm]
8 & 48 & A_{5,6}C_{2,3}A_{1,2} \\[1mm]
9 & 48 & A_{4,5}^{2} \\[1mm]
11 & 48 & A_{6,7} \\[1mm]
14 & 60 & F_{4,6}A_{2,2} \\[1mm]
17 & 72 & A_{5,3}D_{4,3}A_{1,1}^{3} \\[1mm]
20 & 72 & D_{6,5}A_{1,1}^{2} \\[1mm]
21 & 72 & C_{5,3}G_{2,2}A_{1,1} \\[1mm]
27 & 96 & A_{8,3}A_{2,1}^{2} \\[1mm]
28 & 96 & E_{6,4}C_{2,1}A_{2,1} \\[1mm]
32 & 120 & E_{6,3}G_{2,1}^{3} \\[1mm]
34 & 120 & D_{7,3}A_{3,1}G_{2,1} \\[1mm]
45 & 168 & E_{7,3}A_{5,1}
\end{array}
\end{equation*}
Here, $X_{m,n}$ denotes the simple Lie algebra of type $X_{m}$ 
whose ``level'' is equal to $n$ (for the definition of the level, 
see \S\ref{subsec:wt1}).

In \cite{M}, Miyamoto established a $\BZ_{3}$-orbifold 
construction to a lattice VOA and a lattice automorphism of order $3$
satisfying the condition that the rank of the fixed point lattice 
is divisible by $6$ (see \S\ref{subsec:Z3}). 
Then he constructed a new (non-framed) holomorphic VOA 
of central charge $24$ whose Lie algebra of the weight one space 
is of type $E_{6,3}G_{2,1}^3$, by applying his $\BZ_{3}$-orbifold 
construction to the lattice VOA associated to 
the Niemeier lattice $\Ni(E_{6}^{4})$ (with $E_{6}^{4}$ the root lattice) 
and its automorphism $\sigma_{6}$ of order $3$ 
(see \cite[\S5.2]{M} and also Appendix in this paper); 
this holomorphic VOA corresponds to No.\,32 on Schellekens' list. 
Also, he obtained a holomorphic VOA whose weight one space is 
identical to $\bigl\{0\bigr\}$, by applying his $\BZ_{3}$-orbifold 
construction to the Leech lattice VOA and a fixed-point-free 
automorphism $\sigma_{7}$ of order $3$ of the Leech lattice; 
this holomorphic VOA is conjecturally isomorphic to 
the Moonshine VOA $V^{\natural}$. 

In this paper, we also construct some new holomorphic VOAs 
as a further application of Miyamoto's $\BZ_{3}$-orbifold construction; 
in Theorem~\ref{thm:main} (resp., Theorems~\ref{thm:main2}, \ref{thm:main3}, 
\ref{thm:main4}, \ref{thm:main5}), we obtain a holomorphic VOA 
whose Lie algebra of the weight one space is of type $A_{2,3}^{6}$ 
(resp., $A_{2,3}^{6}$, $E_{6,3}G_{2,1}^{3}$, $A_{5,3}D_{4,3}A_{1,1}^{3}$, 
$A_{5,3}D_{4,3}A_{1,1}^{3}$), by applying the $\BZ_{3}$-orbifold construction 
to the lattice VOA associated to the Niemeier lattice $\Ni(A_{2}^{12})$ 
(resp., $\Ni(D_{4}^{6})$, $\Ni(D_{4}^{6})$, $\Ni(D_{4}^{6})$, 
$\Ni(A_{5}^{4}D_{4})$) and its automorphism $\sigma_{1}$ (resp., 
$\sigma_{2}$, $\sigma_{3}$, $\sigma_{4}$, $\sigma_{5}$) of order $3$, 
which corresponds to No.\,6 (resp., No.\,6, No.\,32, No.\,17, No.\,17) 
on Schellekens' list. 

\begin{rem*}
The Lie algebras of the weight one spaces in 
the holomorphic VOAs obtained in Theorems~\ref{thm:main} and \ref{thm:main2}
(resp., Theorem~\ref{thm:main3} and \cite[\S5.2]{M}, 
Theorems~\ref{thm:main4} and ~\ref{thm:main5}) are isomorphic. 
However, we do not know whether or not 
these holomorphic VOAs are isomorphic (as VOAs). 
\end{rem*}

In \cite{ISS}, we will classify, up to conjugation, 
all lattice automorphisms of order $3$ of Niemeier lattices, 
satisfying the condition that 
the ranks of the fixed point lattices are divisible by $6$ 
(i.e., those to which we can apply the $\BZ_{3}$-orbifold construction), 
and prove that if such a lattice automorphism of a Niemeier lattice 
is not conjugate to any of the $\sigma_{1}$ $\sigma_{2}$, $\sigma_{3}$, 
$\sigma_{4}$, $\sigma_{5}$, $\sigma_{6}$, $\sigma_{7}$ above, then 
the holomorphic VOA obtained by the $\BZ_{3}$-orbifold construction is 
isomorphic to the lattice VOA associated to a Niemeier lattice.
In other word, a non-framed holomorphic VOA of central charge $24$ 
which can be obtained by the $\BZ_{3}$-orbifold construction is 
one of those obtained in this paper and \cite{M}. 

\paragraph{Acknowledgments.} 
The authors thank Professor Masahiko Miyamoto 
for variable comments and useful discussions.

%
\section{Review.}
\label{sec:review}

%
\subsection{Lattice VOAs and their automorphisms.}
\label{subsec:lattice}

In this subsection, we review the definition of 
a lattice vertex operator algebra (VOA for short); 
for the details, see, e.g., \cite[\S6.4 and \S6.5]{LL}. 

Let $L$ be a positive-definite, even lattice
with $\BZ$-bilinear form $\pair{\cdot\,}{\cdot}$.
Regard $\Fh:=L \otimes_{\BZ} \BC$ as an abelian Lie algebra, and 
define its affinization to be the Lie algebra 
$\ha{\Fh}:=\Fh \otimes \BC[t,\,t^{-1}] \oplus \BC\bk$ 
with Lie bracket given by: 
\begin{equation*}
[x \otimes t^{m},\,y \otimes t^{n}] 
= \delta_{m+n,\,0} m\pair{x}{y} \bk \quad 
\text{for $x,\,y \in \Fh$ and $m,\,n \in \BZ$},
\end{equation*}
\begin{equation*}
[\ha{\Fh},\,\bk]=\{0\};
\end{equation*}
for simplicity of notation, we denote $h \otimes t^{m}$ by $h(m)$ 
for $h \in \Fh$ and $m \in \BZ$. 
The Lie subalgebra $\ha{\Fb}:=
\Fh \otimes \BC[t] \oplus \BC\bk \subset \ha{\Fh}$ 
acts on the one-dimensional vector space $\BC$ 
as follows: for $c \in \BC$, 
\begin{equation*}
h(m) \cdot c = 0 \quad 
\text{for all $h \in \Fh$ and $m \in \BZ_{\ge 0}$}, \qquad
\bk \cdot c = c.
\end{equation*}
Then we define
\begin{equation*}
M(1):=\Ind^{\ha{\Fh}}_{\ha{\Fb}}\BC. 
\end{equation*}

Fix a positive even integer $s \in 2\BZ_{> 0}$. 
Let us define $c_{0} : L \times L \rightarrow \BZ/s\BZ$ by:
\begin{equation*}
c_{0}(\alpha,\,\beta)=\frac{s}{2} \pair{\alpha}{\beta}+s\BZ,
\end{equation*}
which is an alternating $\BZ$-bilinear map.
Let $\ve_{0} : L \times L \rightarrow \BZ/s\BZ$ 
be a $2$-cocycle corresponding to 
$c_{0} : L \times L \rightarrow \BZ/s\BZ$, 
normalized as: 
$\ve_{0}(\alpha,\,0)=
\ve_{0}(0,\,\alpha)=0$ for all $\alpha \in L$.
Let $\kc$ be the cyclic group of order $s$. 
We define a product on 
\begin{equation*}
\ha{L} := \bigl\{(\kappa^{a},\,e_{\alpha}) \mid 
a \in \BZ/s\BZ,\,\alpha \in L\bigr\}
\end{equation*}
as follows: 
for $a,\,b \in \BZ/s\BZ$ and $\alpha,\,\beta \in L$, 
\begin{equation*}
(\kappa^{a},\,e_{\alpha}) \cdot
(\kappa^{b},\,e_{\beta}):=
(\kappa^{a+b+\ve_{0}(\alpha,\,\beta)},\,e_{\alpha+\beta}).
\end{equation*}
Then, $\ha{L}$ is a group with $(\kappa^{0},\,e_{0})$ 
the identity element, and is 
the central extension of $L$ by the cyclic group $\kc$ of order $s$ 
with $c_{0}:L \times L \rightarrow \BZ/s\BZ$ the commutator map. 
The cyclic group $\kc$ acts on the one-dimensional space $\BC$ by: 
$\kappa \cdot c = \xi c$ for $c \in \BC$, 
where $\xi \in \BC$ is a primitive $s$-th root of unity. 
Define 
\begin{equation*}
\BC\{L\} := \BC[\ha{L}] \otimes_{\kc} \BC, 
\end{equation*}
where $\BC[\ha{L}]$ denotes the group ring of the group $\ha{L}$; 
remark that
$\bigl\{e^{\alpha}:=(\kappa^{0},\,e_{\alpha}) \otimes 1 \mid 
\alpha \in L\bigr\}$ is a basis of $\BC\{L\}$. 

Now, set
\begin{equation*}
V_{L}:=M(1) \otimes \BC\{L\}. 
\end{equation*}
Then, $V_{L}$ admits a VOA structure whose central charge is equal to 
the rank of the lattice $L$ (which is independent of 
the choices of $s$, $\ve_{0}$, and $\xi$). 
Recall that the weight of 
$h_{k}(-n_{k}) \cdots h_{1}(-n_{1}) 1 \otimes e^{\alpha} 
\in V_{L}=M(1) \otimes \BC\{L\}$, where 
$h_{1},\,\dots,\,h_{k} \in \Fh$, 
$n_{1},\,\dots,\,n_{k} \in \BZ_{ > 0}$, and $\alpha \in L$, 
is given by: 
\begin{equation*}
n_{k}+\cdots+n_{1}+\frac{\pair{\alpha}{\alpha}}{2} \in \BZ_{\ge 0}.
\end{equation*}
In particular, the weight one space $(V_{L})_{1}$ of $V_{L}$ 
is spanned by
\begin{equation*}
\bigl\{h(-1)1 \otimes e^{0} \mid h \in \Fh\bigr\} 
\cup \bigl\{1 \otimes e^{\alpha} \mid \alpha \in \Delta\bigr\},
\end{equation*}
where $\Delta=\Delta(L):=
\bigl\{\alpha \in L \mid \pair{\alpha}{\alpha}=2\bigr\}$, 
the set of roots in $L$. 
Denote by 
\begin{equation*}
Y(\cdot\,,\,z):V_{L} \rightarrow (\End_{\BC}V_{L})[[z,\,z^{-1}]], \quad
a \mapsto Y(a,\,z)=\sum_{n \in \BZ} a_{n} z^{-n-1}
\end{equation*}
the vertex operator for $V_{L}$. 
For latter use, let us recall the definition of 
$Y(a,\,z)$ for some special $a \in V_{L}$. 
First, the Lie algebra $\ha{\Fh}$ acts on 
$V_{L}=M(1) \otimes \BC\{L\}$ as follows: 
$\bk \in \ha{\Fh}$ acts as the identity, 
and for $h \in \Fh$ and $n \in \BZ$, 
\begin{equation*}
h(n)(u \otimes e^{\beta})=
\begin{cases}
\pair{h}{\beta}(u \otimes e^{\beta}) & \text{if $n=0$}, \\[1.5mm]
\bigl(h(n)u\bigr) \otimes e^{\beta} & \text{if $n \ne 0$},
\end{cases}
\quad \text{for $u \in M(1)$ and $\beta \in L$}.
\end{equation*}
Also, for $\alpha \in L$, we define 
$z^{\alpha} \in \Hom_{\BC}(V_{L},\,V_{L}[z,\,z^{-1}])$ by
\begin{equation*}
z^{\alpha}(u \otimes e^{\beta}):=
  z^{\pair{\alpha}{\beta}}(u \otimes e^{\beta})
\quad \text{for $u \in M(1)$ and $\beta \in L$}.
\end{equation*}
In addition, the group $\ha{L}$ acts on 
$V_{L}=M(1) \otimes \BC\{L\}$ as follows: 
\begin{equation*}
g \cdot (u \otimes v) = u \otimes (g \cdot v) \quad 
\text{for $g \in \ha{L}$ and $u \in M(1)$, $v \in \BC\{L\}$}, 
\end{equation*}
where the $g \cdot v$ above is given by the natural 
action of $\ha{L}$ on $\BC\{L\}$. In particular, 
\begin{equation*}
(\kappa^{0},\,e_{\alpha}) \cdot (u \otimes e^{\beta}):=
u \otimes \bigl(\xi^{\ve_{0}(\alpha,\,\beta)}e^{\alpha+\beta}\bigr)
\quad \text{for $\alpha,\,\beta \in L$ and $u \in M(1)$}.
\end{equation*}
We have
%
%
\begin{equation} \label{eq:Yh}
Y(h(-1)1 \otimes e^{0},\,z)= \sum_{n \in \BZ} h(n) z^{-n-1}
  \quad \text{for $h \in \Fh$},
\end{equation}
%
%
\begin{equation} \label{eq:Ye}
Y(1 \otimes e^{\alpha},\,z)=
E^{-}(-\alpha,\,z)E^{+}(-\alpha,\,z)
\underbrace{(\kappa^{0},\,e_{\alpha})}_{\in \ha{L}}z^{\alpha} 
\quad \text{for $\alpha \in L$},
\end{equation}
where 
\begin{equation*}
E^{\pm}(-\alpha,\,z):=
\exp\left(\sum_{n \in \pm \BZ_{> 0}} \frac{-\alpha(n)}{n}z^{-n}\right).
\end{equation*}

%
\subsection{Twisted modules over lattice VOAs.}
\label{subsec:twisted}

Keep the notation in \S\ref{subsec:lattice}. 
An automorphism of the lattice $L$ is, by definition, 
a $\BZ$-module automorphism $\sigma$ of $L$ satisfying 
the condition that $\pair{\sigma\alpha}{\sigma\beta}=
\pair{\alpha}{\beta}$ for all $\alpha,\,\beta \in L$. 
Denote by $\Aut(L)$ the group of all lattice automorphisms of $L$. 
Let $\sigma \in \Aut(L)$ be of odd order, and set $s:=2|\sigma|$ 
(with notation in \S\ref{subsec:lattice}), 
where $|\sigma|$ denotes the order $\sigma$. 
Replacing 
$\ve_{0}:L \times L \rightarrow \BZ/s\BZ$ with 
\begin{equation*}
(\alpha,\,\beta) \mapsto 
\sum_{r=0}^{|\sigma|-1} 
\ve_{0}(\sigma^{r}\alpha,\,\sigma^{r}\beta)
\qquad \text{for $\alpha,\,\beta \in L$}
\end{equation*}
if necessary, we may assume that $\ve_{0}$ 
is $\sigma$-invariant.
Then we deduce that 
the lattice automorphism $\sigma \in \Aut(L)$ 
naturally induces a VOA automorphism of $V_{L}$; 
by abuse of notation, we denote this VOA automorphism 
also by $\sigma \in \Aut(V_{L})$. Remark that 
\begin{equation*}
\sigma(h_{k}(-n_{k}) \cdots h_{1}(-n_{1})1 \otimes e^{\alpha})
=(\sigma h_{k})(-n_{k}) \cdots (\sigma h_{1})(-n_{1})1 \otimes 
 e^{\sigma\alpha}.
\end{equation*}

Now, we recall a construction of $\sigma$-twisted modules over 
the lattice VOA $V_{L}$ from \cite{DL} and \cite{L} 
in the case that $\sigma$ is of odd order $p$; 
in fact, a lattice automorphism 
mainly treated in this paper is of order $3$. 
Set $s:=2|\sigma|=2p$ as above, and set $\zeta:=\xi^{2} \in \BC$, 
which is a primitive $p$-th root of unity (recall that 
$\xi$ is a primitive $s$-th root of unity). 

For $n \in \BZ$, define $\Fh_{(n)}:=
\bigl\{h \in \Fh \mid \sigma h=\zeta^{n}h\bigr\}$; 
note that $\Fh_{(n)}=\Fh_{(n+pk)}$ for all $n,\,k \in \BZ$, 
and $\Fh=\Fh_{(0)} \oplus \Fh_{(1)} \oplus \cdots \oplus \Fh_{(p-1)}$. 
Let us define the $\sigma$-twisted affine Lie algebra 
associated to the abelian Lie algebra $\Fh$ to be 
\begin{equation*}
\ha{\Fh}[\sigma] := \bigoplus_{n \in (1/p)\BZ} 
\Fh_{(pn)} \otimes \BC t^{n} \oplus \BC\bk
\end{equation*}
with Lie bracket 
\begin{equation*}
[x \otimes t^{m},\,y \otimes t^{n}] 
= \delta_{m+n,\,0} m \pair{x}{y} \bk \quad
\text{for $m,\,n \in (1/p)\BZ$ and 
$x \in \Fh_{(pm)}$, $y \in \Fh_{(pn)}$},
\end{equation*}
\begin{equation*}
\bigl[ \ha{\Fh}[\sigma],\,\bk \bigr]=\{0\}; 
\end{equation*}
for simplicity of notation, we denote $h \otimes t^{m}$ by $h(m)$ 
for $m \in (1/p)\BZ$ and $h \in \Fh_{(pm)}$. 
The Lie subalgebra 
\begin{equation*}
\ha{\Fb}[\sigma]:=
\bigoplus_{n \in (1/p)\BZ_{\ge 0}} 
\Fh_{(pn)} \otimes \BC t^{n} \oplus \BC\bk \subset \ha{\Fh}[\sigma]
\end{equation*}
acts on the one-dimensional vector space $\BC$ 
as follows: for $c \in \BC$, 
\begin{equation*}
h(m) \cdot c = 0 \quad
\text{for all $m \in (1/p)\BZ_{\ge 0}$ and $h \in \Fh_{(pm)}$}, \qquad
\bk \cdot c = c.
\end{equation*}
Then we define
\begin{equation*}
M(1)[\sigma]:=
\Ind^{\ha{\Fh}[\sigma]}_{\ha{\Fb}[\sigma]}\BC. 
\end{equation*}

Define an alternating $\BZ$-bilinear map 
$c_{0}^{\sigma} : L \times L \rightarrow \BZ/s\BZ$ by:
%
%
\begin{equation} \label{eq:c0s}
c_{0}^{\sigma}(\alpha,\,\beta)=\sum_{r=0}^{p-1} 
(s/2+sr/p) \pair{\sigma^{r}\alpha}{\beta}+s\BZ, 
\end{equation}
and then define $\ve_{0}^{\sigma} : L \times L \rightarrow \BZ/s\BZ$ 
by (see \cite[Remarks~2.1 and 2.2]{DL})
\begin{equation*}
\ve_{0}^{\sigma}(\alpha,\,\beta):=
 \ve_{0}(\alpha,\,\beta) +
\sum_{0 < r < p/2}
(s/2+sr/p) \pair{\sigma^{r}\alpha}{\beta}+s\BZ
\quad \text{for $\alpha,\,\beta \in L$}.
\end{equation*}
It can be easily checked that 
$\ve_{0}^{\sigma}$ is a $\sigma$-invariant, normalized $2$-cocycle 
corresponding to $c_{0}^{\sigma}$. 
We define another product $\ast$ on 
$\ha{L}=\bigl\{(\kappa^{a},\,e_{\alpha}) \mid 
a \in \BZ/s\BZ, \, \alpha \in L\bigr\}$ as follows: 
for $a,\,b \in \BZ/s\BZ$ and $\alpha,\,\beta \in L$, 
\begin{equation*}
(\kappa^{a},\,e_{\alpha}) \ast
(\kappa^{b},\,e_{\beta}):=
(\kappa^{a+b+\ve_{0}^{\sigma}(\alpha,\,\beta)},\,e_{\alpha+\beta}). 
\end{equation*}
Then, $(\ha{L},\,\ast)$ is a group with 
$(\kappa^{0},\,e_{0})$ the identity element; 
we denote this group by $\ha{L}_{\sigma}$. 
It can be easily seen that $\ha{L}_{\sigma}$ is 
the central extension of $L$ by the cyclic group $\kc$ of 
order $s$ with $c_{0}^{\sigma}:L \times L \rightarrow \BZ/s\BZ$ 
the commutator map. 
Because $\ve_{0}^{\sigma}$ is $\sigma$-invariant, 
the lattice automorphism $\sigma \in \Aut(L)$ induces a group automorphism 
$\sigma \in \Aut(\ha{L}_{\sigma})$ defined by: 
$\sigma(\kappa^{a},\,e_{\alpha})=(\kappa^{a},\,e_{\sigma\alpha})$ 
for $a \in \BZ/s\BZ$ and $\alpha \in L$. 

For $h \in \Fh$ and $n \in \BZ$, 
denote by $h_{(n)}$ the image of $h$ 
under the projection $\Fh \twoheadrightarrow \Fh_{(n)}$; 
note that $h_{(0)}=(1/p)\sum_{r=0}^{p-1}\sigma^{r}h$ for $h \in \Fh$.
Set 
%
%
\begin{equation} \label{eq:N}
N:=\bigl\{\alpha \in L \mid \pair{\alpha}{\Fh_{(0)}}=\{0\}\bigr\}=
\bigl\{ \alpha \in L \mid \alpha_{(0)}=0\bigr\};
\end{equation}
the (second) equality follows from the fact that
$\pair{\sigma^{r}\alpha}{h}=\pair{\alpha}{h}$
for all $r \in \BZ$ and all $\alpha \in L$ and $h \in \Fh_{(0)}$, and 
the fact that $\pair{\cdot\,}{\cdot}$ is 
nondegenerate on $\Fh_{(0)}$. Set 
%
%
\begin{equation} \label{eq:RM}
R:=\bigl\{\alpha \in N \mid 
c_{0}^{\sigma}(\alpha,\,N)=0\bigr\}, \quad 
M:=(1-\sigma)L.
\end{equation}
%
%
Then, $M \subset R \subset N$. Also, because 
%
%
\begin{equation} \label{eq:MN}
M \otimes_{\BZ} \BC = 
 (1-\sigma) \underbrace{(L \otimes_{\BZ} \BC)}_{=\Fh} = 
 \Fh_{(1)} \oplus \Fh_{(2)}=
 \bigl\{h \in \Fh \mid \pair{h}{\Fh_{(0)}}=\{0\}\bigr\}=N \otimes_{\BZ} \BC,
\end{equation}
it follows immediately that $N/M$ is a finite group, 
and hence so is $N/R$. 

Now, for a sublattice $Q$ of $L$, we set
$\ha{Q}_{\sigma}:=
\bigl\{(\kappa^{a},\,e_{\alpha}) \mid a \in \BZ/s\BZ,\,\alpha \in Q\bigr\}$ 
(which is a subgroup of $\ha{L}_{\sigma}$). 
Then, $\ha{M}_{\sigma} \subset \ha{R}_{\sigma} \subset 
\ha{N}_{\sigma}$. 
It is known from \cite[Propositions~6.1 and 6.2]{L} 
(see also \cite[Remark~4.2]{DL}) that 
there exists a finite-dimensional irreducible $\ha{N}_{\sigma}$-module $T$ 
of dimension $|N/R|^{1/2}$ on which $\ha{M}_{\sigma}$ acts according 
to a group homomorphism $\tau:\ha{M}_{\sigma} \rightarrow 
\langle \xi \rangle \subset \BC^{\times}$, 
i.e., $g \cdot t = \tau(g)t$ for $g \in \ha{M}_{\sigma}$ and $t \in T$. 
Set 
\begin{equation*}
U_{T}:=\Ind^{\ha{L}_{\sigma}}_{\ha{N}_{\sigma}} T.
\end{equation*}
%
%
\begin{rem}[{see \cite[Section 7]{L}}] \label{rem:NT}
Since $\tau(\ha{M}_{\sigma}) \subset \langle \xi \rangle$, 
it follows immediately that $\ha{M}_{\sigma}/\ker \tau$ is 
a finite group. Also, since $N/M$ is a finite group as seen above, 
so is $\ha{N}_{\sigma}/\ha{M}_{\sigma}$. Thus, 
$\ha{N}_{\sigma}/\ker \tau$ is a finite group. 
Since every element in $\ker \tau$ 
acts on $T$ as the identity, we conclude that 
each $g \in \ha{N}_{\sigma}$ acts on $T$ as a linear automorphism 
of finite order. In particular, the action of 
each $g \in \ha{N}_{\sigma}$ on $T$ is semisimple. 
\end{rem}

Now, we define
\begin{equation*}
V_{L}^{T}:=M(1)[\sigma] \otimes U_{T}.
\end{equation*}
Then we know from \cite[Theorem~7.1]{DL} that 
$V_{L}^{T}$ admits a $\sigma$-twisted $V_{L}$-module structure. 
Note that $V_{L}^{T}=M(1)[\sigma] \otimes U_{T}$ is spanned by the elements 
of the form: $h_{k}(-n_{k}) \cdots h_{1}(-n_{1}) 1 \otimes (g \cdot t)$ 
with $n_{1},\,\dots,\,n_{k} \in (1/p)\BZ_{> 0}$, 
$h_{1} \in \Fh_{(-pn_{1})},\,\dots,\,h_{k} \in \Fh_{(-pn_{k})}$, 
and $g \in \ha{L}_{\sigma}$, $t \in T$. 
The weight of an element of the form above is 
given by 
%
%
\begin{equation} \label{eq:wt}
n_{k}+\cdots+n_{1}+\rho+
\frac{1}{2}\pair{ \ol{g}_{(0)} }{ \ol{g}_{(0)} } \in \rho + (1/p)\BZ_{\ge 0},
\end{equation}
where 
%
%
\begin{equation} \label{eq:def-rho}
\rho:=
\frac{1}{4p^{2}}\sum_{r=1}^{p-1} r(p-r) \dim \Fh_{(r)},
\end{equation}
and where for $g=(\kappa^{a},\,e_{\alpha}) \in \ha{L}_{\sigma}$, 
we set $\ol{g}:=\alpha \in L$; notice that for $g \in \ha{L}_{\sigma}$, 
\begin{equation*}
\pair{\ol{g}_{(0)}}{\ol{g}_{(0)}}=0 \iff
\ol{g}_{(0)}=0 \iff \ol{g} \in N \iff g \in \ha{N}_{\sigma},
\end{equation*}
and hence 
%
%
\begin{equation} \label{eq:top}
(V_{L}^{T})_{\rho} = \BC 1 \otimes T \quad 
\text{with} \quad \dim (V_{L}^{T})_{\rho} = |N/R|^{1/2}.
\end{equation}
Denote by 
\begin{equation*}
Y_{\sigma}(\cdot\,,\,z):
 V_{L} \rightarrow 
 (\End_{\BC} V_{L}^{T})[[z^{1/p},\,z^{-1/p}]], \quad 
a \mapsto Y_{\sigma}(a,\,z)=\sum_{n \in (1/p)\BZ} a_{n} z^{-n-1}
\end{equation*}
the $\sigma$-twisted vertex operator for $V_{L}^{T}$. 
For latter use, let us recall the definition of 
$Y_{\sigma}(a,\,z)$ for some special $a \in V_{L}$. 
First, the Lie algebra $\ha{\Fh}[\sigma]$ acts on 
$V_{L}^{T}=M(1)[\sigma] \otimes U_{T}$ as follows: 
$\bk$ acts as the identity, and 
for $n \in (1/p)\BZ$ and $h \in \Fh_{(pn)}$, 
%
%
\begin{equation} \label{eq:hn}
h(n)(u \otimes (g \cdot t))=
\begin{cases}
\pair{h}{\ol{g}_{(0)}} (u \otimes (g \cdot t)) & \text{if $n=0$}, \\[1.5mm]
\bigl(h(n)u\bigr) \otimes (g \cdot t) & \text{if $n \ne 0$}
\end{cases}
\end{equation}
for $u \in M(1)[\sigma]$ and $g \in \ha{L}_{\sigma}$, $t \in T$.
For $\alpha \in L$, define 
$z^{\alpha_{(0)}} \in 
\Hom_{\BC}(V_{L}^{T},\,V_{L}^{T}[z^{1/p},\,z^{-1/p}])$ by
\begin{equation*}
z^{\alpha_{(0)}}(u \otimes (g \cdot t)):=
  z^{ \pair{\alpha_{(0)}}{\ol{g}_{(0)}} }(u \otimes (g \cdot t))
\quad \text{for $u \in M(1)[\sigma]$ and $g \in \ha{L}_{\sigma}$, $t \in T$}.
\end{equation*}
In addition, the group $\ha{L}_{\sigma}$ 
acts on $V_{L}^{T}$ as follows: 
\begin{equation*}
g \cdot (u \otimes w) = u \otimes (g \cdot w) \quad 
\text{for $x \in \ha{L}_{\sigma}$ and $u \in M(1)[\sigma]$, $w \in U_{T}$}. 
\end{equation*}
Now, we deduce from \cite[(4.40) and (4.45)]{DL} that 
%
%
\begin{equation} \label{eq:Ysh}
Y_{\sigma}(h(-1)1 \otimes e^{0},\,z) = 
\sum_{n \in (1/p)\BZ} h_{(pn)}(n) z^{-n-1};
\end{equation}
observe that $\varDelta_{z}(h(-1)1 \otimes e^{0})=0$, 
where $\varDelta_{z}$ is defined as \cite[(4.42)]{DL}. 
Also, we know from 
\cite[(4.34) and (4.39)]{DL} that for $\alpha \in L$, 
the vertex operator $Y_{\sigma}(1 \otimes e^{\alpha},\,z)$ 
is equal, up to a specified constant multiple 
(which depends only on $\alpha$), to 
%
%
\begin{equation} \label{eq:Yse}
E^{-}_{\sigma}(-\alpha,\,z)E^{+}_{\sigma}(-\alpha,\,z)
\underbrace{(\kappa^{0},\,e_{\alpha})}_{\in \ha{L}_{\sigma}}
z^{ \alpha_{(0)}+
\{\pair{\alpha_{(0)}}{\alpha_{(0)}}-
\pair{\alpha}{\alpha}\}/2 }, 
\end{equation}
where
\begin{equation*}
E^{\pm}_{\sigma}(-\alpha,\,z):=
\exp\left(\sum_{n \in \pm (1/p)\BZ_{> 0}} 
\frac{-\alpha_{(pn)}(n)}{n}z^{-n}\right).
\end{equation*}

Recall that for every $a \in (V_{L})_{1}$, the $0$-th operator $a_{0} 
\in \End_{\BC}(V_{L}^{T})$ (i.e., the coefficient of $z^{-1}$ in 
$Y_{\sigma}(a,\,z)$) is weight-preserving. In particular, 
the top weight space $(V_{L}^{T})_{\rho}$ is stable under 
the action of $a_{0}$. 

%
\begin{lem} \label{lem:Ysh}
{\rm (1)} For every $h \in \Fh$, 
the $0$-th operator 
$(h(-1)1 \otimes e^{0})_{0} \in \End_{\BC}(V_{L}^{T})$ of 
$h(-1)1 \otimes e^{0} \in (V_{L})_{1}$ acts on 
the top weight space $(V_{L}^{T})_{\rho}$ trivially. 

{\rm (2)} If $\alpha \in \Delta \setminus N$, then 
the $0$-th operator $(1 \otimes e^{\alpha})_{0} 
\in \End_{\BC}(V_{L}^{T})$ of 
$1 \otimes e^{\alpha} \in (V_{L})_{1}$
acts on the top weight space $(V_{L}^{T})_{\rho}$ trivially. 

{\rm (3)} If $\alpha \in \Delta \cap N$, then 
the action of the $0$-th operator $(1 \otimes e^{\alpha})_{0} 
\in \End_{\BC}(V_{L}^{T})$ of 
$1 \otimes e^{\alpha} \in (V_{L})_{1}$ on $(V_{L}^{T})_{\rho}$ 
is semisimple. 
\end{lem}

\begin{proof}
(1) By \eqref{eq:Ysh}, we have 
$(h(-1)1 \otimes e^{0})_{0}=h_{(0)}(0)$. 
Because $(V_{L}^{T})_{\rho}=\BC 1 \otimes T$ by \eqref{eq:top}, 
it follows immediately from \eqref{eq:hn} that 
$(h(-1)1 \otimes e^{0})_{0}=h_{(0)}(0)$ 
acts on the top weight space $(V_{L}^{T})_{\rho}$ trivially. 

(2), (3) Let $\alpha \in \Delta$, and 
let $1 \otimes t \in (V_{L}^{T})_{\rho}=\BC 1 \otimes T$. 
Set $v:=(\kappa^{0},\,e_{\alpha}) \cdot t \in U_{T}$, and 
$d:=\pair{\alpha_{(0)}}{\alpha_{(0)}}/2 \in (1/p)\BZ$; 
remark that $d=0$ if and only if $\alpha \in N$. 
By \eqref{eq:Yse}, 
$(1 \otimes e^{\alpha})_{0}(1 \otimes t)$ is 
a scalar multiple of the coefficient of $z^{-1}$ in 
\begin{align*}
& E^{-}_{\sigma}(-\alpha,\,z)E^{+}_{\sigma}(-\alpha,\,z)
(\kappa^{0},\,e_{\alpha})
z^{ \alpha_{(0)}+
\{\pair{\alpha_{(0)}}{\alpha_{(0)}}-
\pair{\alpha}{\alpha}\}/2 } (1 \otimes t) \\
& \qquad
= E^{-}_{\sigma}(-\alpha,\,z)E^{+}_{\sigma}(-\alpha,\,z)
  (1 \otimes v \bigr)z^{-1+d} \\
& \qquad
= E^{-}_{\sigma}(-\alpha,\,z)
  (1 \otimes v \bigr)z^{-1+d} \\
& \hspace*{30mm}
\quad \text{since $\alpha_{(pn)}(n)1=0$ for all $n \in (1/p)\BZ_{>0}$} \\
& \qquad
= (1 \otimes v)z^{-1+d} + 
  \text{(higher terms)}. 
\end{align*}
If $\alpha \notin N$, then 
the coefficient of $z^{-1}$ in 
$Y_{\sigma}(1 \otimes e^{\alpha},\,z)(1 \otimes t)$ is 
equal to $0$ (since $d > 0$), which implies that 
$(1 \otimes e^{\alpha})_{0}(1 \otimes t)=0$. 
Thus we have proved part (2). 
Assume that $\alpha \in N$. Then, 
$(\kappa^{0},\,e_{\alpha}) \in \ha{N}_{\sigma}$, 
and $(1 \otimes e^{\alpha})_{0}$ sends $1 \otimes t 
\in (V_{L}^{T})_{\rho}=\BC 1 \otimes T$ to a scalar multiple of 
$1 \otimes v=1 \otimes \bigl((\kappa^{0},\,e_{\alpha}) \cdot t\bigr)$.
Therefore, by Remark~\ref{rem:NT}, 
$(1 \otimes e^{\alpha})_{0}$ is semisimple on $(V_{L}^{T})_{\rho}$. 
Thus we have proved part (3). 
\end{proof}

The following lemma may be known, 
but we give a proof for completion. 
%
%
\begin{lem} \label{lem:irred}
The $\sigma$-twisted $V_{L}$-module $V_{L}^{T}=
M(1)[\sigma] \otimes U_{T}$ is irreducible.
\end{lem}

\begin{proof}
Let $W \subset V_{L}^{T}$ be a nonzero $\sigma$-twisted $V_{L}$-submodule. 
First, we show that $W$ contains a nonzero element of the form: 
$u \otimes (g \cdot t)$ for some $u \in M(1)[\sigma]$ 
and $g \in \ha{L}_{\sigma}$, $t \in T$. 
Take a complete set $\bigl\{g_{i} \mid i \in I\bigr\} 
\subset \ha{L}_{\sigma}$ of representatives for cosets 
in $\ha{L}_{\sigma}/\ha{N}_{\sigma}$. Then, 
\begin{equation*}
U_{T}=\bigoplus_{i \in I} g_{i} \cdot T.
\end{equation*}
Let $w \in W$, $w \ne 0$. There exists a finite subset $J$ of $I$ 
such that $w= \sum_{i \in J} u_{i} \otimes (g_{i} \cdot t_{i})$
with some $u_{i} \in M(1)[\sigma]$ and $t_{i} \in T$ for $i \in J$.
For each $h \in \Fh_{(0)}$, we have
%
%
\begin{equation} \label{eq:irred1}
h(0)^{k}w=
\sum_{i \in J} \pair{h}{\ol{g_{i}}_{(0)}}^{k}
\bigl(u_{i} \otimes (g_{i} \cdot t_{i})\bigr) \qquad 
\text{for $0 \le k \le |J|-1$}. 
\end{equation}
Because $\ol{g_{i}}_{(0)}$, 
$i \in I$, are all distinct (notice that 
$\ol{g_{i}}_{(0)}=\ol{g_{j}}_{(0)} \iff 
\ol{g_{i}}- \ol{g_{j}} \in N \iff 
g_{i}g_{j}^{-1} \in \ha{N}_{\sigma}$), 
we can take $h \in \Fh_{(0)}$ in such a way that 
$\pair{h}{\ol{g_{i}}_{(0)}}$, $i \in J$, 
are all distinct. Then the coefficient matrix of 
equation system \eqref{eq:irred1} (which is a 
Vandermonde type matrix) is invertible. Therefore, 
for each $i \in J$, $u_{i} \otimes (g_{i} \cdot t_{i})$ 
can be written as a linear combination of $h(0)^{k}w$ for
$0 \le k \le |J|-1$. Since $h(0)=(h(-1)1 \otimes e^{0})_{0}$ 
by \eqref{eq:Ysh}, and since $W$ is a $\sigma$-twisted $V_{L}$-submodule 
by assumption, we get
$u_{i} \otimes (g_{i} \cdot t_{i}) \in W$ for every $i \in J$. 

Next, let us show that $W$ includes the top weight space 
$(V_{L}^{T})_{\rho}=\BC 1 \otimes U_{T}$. Take 
$g \in \ha{L}_{\sigma}$ and $t \in T$ such that 
\begin{equation*}
W \cap \bigl(M(1)[\sigma] \otimes (g \cdot t)\bigr) \ne \{0\}.
\end{equation*}
By virtue of \eqref{eq:hn} and \eqref{eq:Ysh}, 
both of $W$ and $M(1)[\sigma] \otimes (g \cdot t)$ are 
$\ha{\Fh}[\sigma]$-modules. Furthermore, 
we deduce by standard argument 
(as for the Fock space over the Heisenberg algebra) that 
$M(1)[\sigma] \otimes (g \cdot t)$ is 
an irreducible $\ha{\Fh}[\sigma]$-module. Thus, 
$W \cap \bigl(M(1)[\sigma] \otimes (g \cdot t)\bigr)= 
M(1)[\sigma] \otimes (g \cdot t)$, which implies that
\begin{equation*}
W \supset M(1)[\sigma] \otimes (g \cdot t).
\end{equation*}
In particular, $W \supset \BC1 \otimes (g \cdot t)$.  
Now, for $g'=(\kappa^{a},\,e_{\alpha}) \in \ha{L}_{\sigma}$, 
we deduce, as in the proof of part (2), (3) of Lemma~\ref{lem:Ysh}, 
that 
\begin{equation*}
Y(1 \otimes e^{\alpha},\,z)(\underbrace{1 \otimes (g \cdot t)}_{\in W}) 
= C (1 \otimes (g'g \cdot t)) z^{d} + \text{(higher terms)}
\end{equation*}
for some $C \in \BC^{\times}$ and $d \in (1/p)\BZ$; 
recall that $(\kappa,\,e_{0}) \in \ha{L}_{\sigma}$ acts on 
$U_{T}$ as a scalar multiple by $\xi$. Hence, 
$1 \otimes (g'g \cdot t) \in W$ for every $g' \in \ha{L}_{\sigma}$, 
which implies that $1 \otimes U_{T} \subset W$. 

Finally, since $W$ is an $\ha{\Fh}[\sigma]$-module as mentioned above, 
it follows immediately that $M(1)[\sigma] \otimes U_{T} \subset W$, 
and hence $W=M(1)[\sigma] \otimes U_{T}$. We have thus proved the lemma. 
\end{proof}

%
\subsection{Miyamoto's $\BZ_{3}$-orbifold construction.}
\label{subsec:Z3}

Keep the notation in the previous subsections 
\S\ref{subsec:lattice} and \S\ref{subsec:twisted}. 
Assume, in addition, that 

(i) $L$ is unimodular, and $\sigma \in \Aut(L)$ is of order $3$; 

(ii) $\rho=\frac{1}{18}(\dim \Fh_{(1)} + \dim \Fh_{(2)}) \in (1/3)\BZ$ 
(see \cite[\S5]{M}).

\noindent
Since $L$ is unimodular by assumption (i), 
the lattice VOA $V_{L}$ is holomorphic. 
Also we know from \cite[Theorem 10.3]{DLM00} that 
for each $r=1,\,2$, there exists a unique irreducible 
$\sigma^{r}$-twisted $V_{L}$-module, 
which we denote by $V_{L}(\sigma^{r})$; 
by Lemma~\ref{lem:irred}, these twisted $V_{L}$-modules 
can be obtained by the method (due to Dong and Lepowsky) 
mentioned in \S\ref{subsec:twisted}. 
Recall from \eqref{eq:wt} 
that $V_{L}(\sigma)$ and $V_{L}(\sigma^{2})$ decompose as follows:
\begin{equation*}
V_{L}(\sigma)=\bigoplus_{n \in (1/3)\BZ_{\ge 0}} V_{L}(\sigma)_{\rho+n}, \qquad 
V_{L}(\sigma^{2})=\bigoplus_{n \in (1/3)\BZ_{\ge 0}} V_{L}(\sigma^{2})_{\rho+n}.
\end{equation*}
%
%
\begin{rem} \label{rem:dual}
We see from \cite[Lemma 3.7]{DLM98} that the restricted dual 
\begin{equation*}
V_{L}(\sigma)':=
 \bigoplus_{n \in (1/3)\BZ_{\ge 0}} V_{L}(\sigma)_{\rho+n}^{\ast}
\end{equation*}
of $V_{L}(\sigma)$ admits an irreducible $\sigma^{2}$-twisted 
$V_{L}$-module structure. By uniqueness, 
$V_{L}(\sigma^{2})$ is isomorphic to $V_{L}(\sigma)'$.
\end{rem}

Recall that $\rho \in (1/3)\BZ_{\ge 0}$ 
by assumption (ii). Set 
\begin{equation*}
V_{L}(\sigma)_{\BZ}=
  \bigoplus_{n \in \BZ} V_{L}(\sigma)_{n}, \qquad 
V_{L}(\sigma^{2})_{\BZ}=
  \bigoplus_{n \in \BZ} V_{L}(\sigma^{2})_{n}, 
\end{equation*}
and then
\begin{equation*}
\ti{V}_{L}^{\sigma}:=
 V_{L}^{\sigma} \oplus V_{L}(\sigma)_{\BZ} \oplus V_{L}(\sigma^{2})_{\BZ},
\end{equation*}
where $V_{L}^{\sigma}$ is the fixed point subVOA of $V_{L}$ 
under $\sigma \in \Aut(V_{L})$. 
%
%
\begin{thm}[{\cite[\S5]{M}}] \label{thm:M}
We can give $\ti{V}_{L}^{\sigma}$ a VOA structure 
whose central charge is equal to the rank of the lattice $L$.
Furthermore, the VOA $\ti{V}_{L}^{\sigma}$ is $C_{2}$-cofinite and 
holomorphic. 
\end{thm}
%
%
\begin{rem} \label{rem:SCE}
In fact, the VOA $\ti{V}_{L}^{\sigma}$ is 
a $(\BZ/3\BZ)$-graded simple current extension 
of $V_{L}^{\sigma}$; for the definition and 
properties of simple current extensions, 
see, e.g., \cite[\S2]{LY}. 
\end{rem}

Denote by 
\begin{equation*}
\ti{Y}(\cdot\,,\,z):
 \ti{V}_{L}^{\sigma} \rightarrow 
 (\End_{\BC} \ti{V}_{L}^{\sigma})[[z,\,z^{-1}]], \quad 
a \mapsto \ti{Y}(a,\,z)=\sum_{n \in \BZ} a_{n} z^{-n-1}
\end{equation*}
the vertex operator for the VOA $\ti{V}_{L}^{\sigma}$. 
Then, for $a \in V_{L}^{\sigma}$, 
%
%
\begin{equation} \label{eq:Yti}
\ti{Y}(a,\,z)=
\begin{cases}
Y(a,\,z) & \text{on $V_{L}^{\sigma}$}, \\[1.5mm]
Y_{\sigma}(a,\,z) & \text{on $V_{L}(\sigma)_{\BZ}$}, \\[1.5mm]
Y_{\sigma^{2}}(a,\,z) & \text{on $V_{L}(\sigma^{2})_{\BZ}$}.
\end{cases}
\end{equation}

%
\subsection{Lie algebra of the weight one space.}
\label{subsec:wt1}
First, let us recall the following basic fact: 
Let $V=\bigoplus_{n \in \BZ_{\ge 0}} V_{n}$ be a (general) VOA 
with $a \in V \mapsto Y(a,\,z)=\sum_{n \in \BZ} a_{n}z^{-n-1} 
\in (\End_{\BC} V)[[z,\,z^{-1}]]$ the vertex operator. 
If $\dim V_{0}=1$, then the weight one space $V_{1}$ of $V$ 
carries a Lie algebra structure with Lie bracket 
given by: $[a,\,b]=a_{0}b$ for $a,\,b \in V_{1}$. 
%
%
\begin{rem} \label{rem:Z3}
Keep the notation and setting in \S\ref{subsec:Z3}. 
We see that the VOA $\ti{V}_{L}^{\sigma}$ in Theorem~\ref{thm:M} 
satisfies the condition that $\dim (\ti{V}_{L}^{\sigma})_{0} = 1$. 
The Lie algebra $(\ti{V}_{L}^{\sigma})_{1}$ of the weight one space 
decomposes as:
\begin{equation*}
(\ti{V}_{L}^{\sigma})_{1} = 
(V_{L}^{\sigma})_{1} \oplus V(\sigma)_{1} \oplus V(\sigma^{2})_{1}.
\end{equation*}
This decomposition makes $(\ti{V}_{L}^{\sigma})_{1}$ 
a $(\BZ/3\BZ)$-graded Lie algebra, where 
$(V_{L}^{\sigma})_{1}$ (resp., $V(\sigma)_{1}$, 
$V(\sigma^{2})_{1}$) is the subspace of degree $\ol{0} \in \BZ/3\BZ$ 
(resp., $\ol{1},\,\ol{2} \in \BZ/3\BZ$).
In particular, $(V_{L}^{\sigma})_{1}$ is a Lie subalgebra of 
$(\ti{V}_{L}^{\sigma})_{1}$, and each of 
$V_{L}(\sigma)_{1}$ and $V_{L}(\sigma^{2})_{1}$ is 
a module over $(V_{L}^{\sigma})_{1}$ via the adjoint action 
(see also \eqref{eq:Yti}). 
\end{rem}

The next proposition follows immediately 
from \cite[Theorem 3]{DM04}. 
%
%
\begin{prop} \label{prop:DM}
Let $V=\bigoplus_{n \in \BZ_{\ge 0}} V_{n}$ be 
a $C_{2}$-cofinite, holomorphic VOA of central charge $24$ satisfying 
the condition that $\dim V_{0}=1$ 
(such as $\ti{V}_{L}^{\sigma}$ in Theorem~\ref{thm:M}). 
Then the Lie algebra $V_{1}$ of the weight one space is 
$\bigl\{0\bigr\}$, abelian (of dimension $24$), or semisimple.
\end{prop}

Keep the setting in Proposition~\ref{prop:DM}. 
When the Lie algebra $V_{1}$ of the weight one space is semisimple,
we define the level of a simple component of $V_{1}$ as follows 
(see also, e.g., \cite[\S3]{DM04} and \cite[\S3]{DM06}): 
Let $\Fs \subset V_{1}$ be a simple component of $V_{1}$, and 
$(\cdot\,,\,\cdot)_{\Fs}$ the nondegenerate, symmetric, 
invariant bilinear form on $\Fs$, 
normalized so that the squared length of 
a long root of $\Fs$ is equal to $2$. 
Then, for $x,\,y \in \Fs$ and $m,\,n \in \BZ$, we have
\begin{equation*}
[x_{m},\,y_{n}]= 
(x_{0}y)_{m+n}+k_{\Fs}m\delta_{m+n,\,0}(x,\,y)_{\Fs}\id_{V} \quad \text{on $V$}
\end{equation*}
for some $k_{\Fs} \in \BZ_{\ge 1}$ (see \cite[Theorem 3.1\,(c)]{DM06}); 
we call $k_{\Fs}$ the level of $\Fs$, and say that $\Fs$ is 
of type $X_{\ell,\,k_{\Fs}}$ if $\Fs$ is 
the simple Lie algebra of type $X_{\ell}$.
We know from \cite[(1.1)]{DM04} that 
%
%
\begin{equation} \label{eq:level}
\frac{h_{\Fs}^{\vee}}{k_{\Fs}}=\frac{\dim V_{1}-24}{24}
\end{equation}
for every simple component $\Fs$ of $V_{1}$, where 
$h_{\Fs}^{\vee}$ is the dual Coxeter number of 
the simple Lie algebra $\Fs$.  

In the following, we denote by $\Fg_{X_{\ell}}$ 
the simple Lie algebra of type $X_{\ell}$. 
%
%
\begin{rem} \label{rem:wt1}
A Niemeier lattice is, by definition, a unimodular, positive-definite, 
even lattice of rank $24$; for the list of all Niemeier lattices, 
see \cite[Chapter 16, Table 16.1]{CS}. Let $L$ be a Niemeier lattice.
Then the lattice VOA $V_{L}$ is $C_{2}$-cofinite, holomorphic, and 
of central charge $24$. 
Let $Q$ be the root lattice of $L$, i.e., 
the sublattice of $L$ generated by the roots $\Delta(L)$. 
If $Q=\bigl\{0\bigr\}$, then $(V_{L})_{1}$ is an abelian Lie algebra 
of dimension $24$. Assume that $Q$ is not identical to 
$\bigl\{0\bigr\}$, and decomposes into the direct sum of 
indecomposable root lattices as follows: 
$Q=Q_{1} \oplus \cdots \oplus Q_{m}$. Then, 
$(V_{L})_{1} \cong \Fs_{1} \oplus \cdots \oplus \Fs_{m}$ 
as Lie algebras, where $\Fs_{k}$, $1 \le k \le m$, is 
the simple Lie algebra whose type is same with that of $Q_{k}$. 
In addition, the level of $\Fs_{k}$ is equal to $1$ 
for every $1 \le k \le m$. 
\end{rem}

%
\section{Construction of a holomorphic VOA (1).}
\label{sec:hvoa1}

%
\subsection{Root lattice $A_{n}$.}
\label{subsec:An}

Following \cite[Chapter 4, \S6.1]{CS}, we set
\begin{equation*}
A_{n}:=\bigl\{ 
(x_{0},\,x_{1},\,\dots,\,x_{n}) \in \BZ^{n+1} 
 \mid x_{0}+x_{1}+ \cdots + x_{n} =0
\bigr\},
\end{equation*}
\begin{equation*}
\Delta(A_{n}):=\bigl\{\alpha \in A_{n} \mid 
\pair{\alpha}{\alpha}=2\bigr\} \quad 
\text{(the set of roots in $A_{n}$)},
\end{equation*}
\begin{equation*}
[\ell] : = 
\frac{1}{n+1}(\ell,\,\dots,\,\ell,\,
\underbrace{\ell-n-1,\,\dots,\,\,\ell-n-1}_{%
 \text{$\ell$ times}}) \in A_{n}^{\ast}
\qquad \text{for $\ell=0,\,1,\,\dots,\,n$}, 
\end{equation*}
where $A_{n}^{\ast} \subset A_{n} \otimes_{\BZ} \BR$ 
denotes the dual lattice of $A_{n}$. 
If we set $\ol{[\ell]} : = [\ell]+A_{n}$ for 
$\ell=0,\,1,\,\dots,\,n$, then 
$A_{n}^{\ast}/A_{n}=\bigl\{\ol{[\ell]} \mid 
\ell=0,\,1,\,\dots,\,n\bigr\}$, and 
there exists an isomorphism of additive groups
from $A_{n}^{\ast}/A_{n}$ to $\BZ/(n+1)\BZ=
\bigl\{\ol{\ell} \mid \ell=0,\,1,\,\dots,\,n\bigr\}$ that 
maps $\ol{[\ell]}$ to $\ol{\ell}$ for $\ell=0,\,1,\,\dots,\,n$. 
Furthermore, we define an action of the ring $\BZ/(n+1)\BZ$ on $A_{n}^{\ast}/A_{n}$ 
in such a way that the isomorphism 
$A_{n}^{\ast}/A_{n} \stackrel{\sim}{\rightarrow} \BZ/(n+1)\BZ$ above 
is also an isomorphism of $\BZ/(n+1)\BZ$-modules. Namely, we set 
\begin{equation*}
\ol{\ell_1} \cdot \ol{[\ell_2]} := \ol{[\ell_1 \ell_2 \mod n+1]} 
\quad \text{for $\ell_{1},\,\ell_{2}=0,\,1,\,\dots,\,n$}. 
\end{equation*}
%

%
\subsection{Niemeier lattice $\Ni(A_{2}^{12})$ and 
its automorphism $\sigma_{1}$ of order $3$.}
\label{subsec:Nie}

 In this subsection, 
we recall the definition of the Niemeier lattice $\Ni(A_{2}^{12})$ 
with $A_{2}^{12}$ the root lattice, and define a lattice automorphism 
$\sigma_{1} \in \Aut(\Ni(A_{2}^{12}))$ of order $3$. 

Define $Q$ to be the direct sum $A_{2}^{12}$ of 
12 copies of the root lattice $A_{2}$; 
following \cite[Chapter 10, \S1.5]{CS}, 
we use $\Omega:=\bigl\{\infty,\,0,\,1,\,\dots,\,10\bigr\}$ 
for the index set of the coordinate for $Q$, that is, 
$Q=\bigl\{(\alpha_{i})_{i \in \Omega} \mid 
  \text{$\alpha_{i} \in A_{2}$ for $i \in \Omega$}
  \bigr\}$.
Since $Q^{\ast}=(A_{2}^{\ast})^{12}$, it follows from \S\ref{subsec:An} that
\begin{equation*}
Q^{\ast}/Q=
 \bigl\{(\ol{[\ell_{i}]})_{i \in \Omega} \mid 
   \text{$\ell_{i}=0,\,1,\,2$ for each $i \in \Omega$}
 \bigr\},
\end{equation*}
which we can identify with the $12$-dimensional 
vector space $\BF_{3}^{12}$ over the field $\BF_{3}:=\BZ/3\BZ$ 
of three elements. 
For $\ol{[\ell]} \in A_{2}^{\ast}/A_{2}$ and $i \in \Omega$, 
define $\ol{[\ell]}^{(i)}$ to be the element 
$(\ol{[\ell_{i}]})_{i \in \Omega} \in Q^{\ast}/Q$ 
with $\ell_{i}=\ell$ and $\ell_{j}=0$ 
for all $j \in \Omega$, $j \ne i$. 
Then, $\bigl\{\ol{[1]}^{(i)} \mid i \in \Omega\bigr\}$ 
forms a basis of $Q^{\ast}/Q$, which we identify with 
the standard orthonormal basis of $\BF_{3}^{12}$. 

Denote by $\FS(\Omega)$ the symmetric group of $\Omega$; 
each element $g \in \FS(\Omega)$ acts linearly on 
the vector space $Q^{\ast}/Q$ by: $g \cdot \ol{[1]}^{(i)}=
\ol{[1]}^{(g(i))}$ for $i \in \Omega$. Define $\nu \in \FS(\Omega)$ by: 
$\nu=(\infty)(\ten 9876543210)$, where $\ten$ denotes $10$, 
that is, 
\begin{equation*}
\nu(i)=
\begin{cases}
10 & \text{if $i=0$}, \\[1.5mm]
i-1 & \text{if $1 \le i \le 10$}, \\[1.5mm]
\infty & \text{if $i=\infty$}.
\end{cases}
\end{equation*}
Set $\Theta:=\bigl\{0,\,1,\,3,\,4,\,5,\,9\bigr\} \subset \Omega$, 
and define 
\begin{equation*}
w_{0}:=\sum_{i \in \Omega \setminus \Theta}\ol{[1]}^{(i)}- 
\sum_{j \in \Theta}\ol{[1]}^{(j)}=
\sum_{i \in \Omega \setminus \Theta}\ol{[1]}^{(i)}+
2\sum_{j \in \Theta}\ol{[1]}^{(j)},
\end{equation*}
\begin{equation*}
w_{i}:=\nu^{i} \cdot w_{0} \quad \text{for $0 \le i \le 10$}, 
\qquad
w_{\infty}:=\sum_{i \in \Omega}\ol{[1]}^{(i)}. 
\end{equation*}
%
%
\begin{thm}[{\cite[Chapter 10, Theorems 2 and 3]{CS}}] \label{thm:Golay}
{\rm (1)} Define $\CC_{12}$ to be the subspace of $Q^{\ast}/Q$ 
spanned by $\bigl\{w_{i} \mid i \in \Omega\bigr\}$. 
Then, $\CC_{12}$ is isomorphic to the ternary Golay code in $\BF_{3}^{12}$.

{\rm (2)} The subspace $\CC_{12}$ is $6$-dimensional 
with $\bigl\{w_\infty,\,w_1,\,w_3,\,w_4,\,w_5,\,w_9\bigr\}$ a basis. 

{\rm (3)} The subspace $\CC_{12}$ is stable under the action of 
$\nu \in \FS(\Omega)$. 

{\rm (4)} Define $\delta \in \FS(\Omega)$ by: 
$\delta=(\infty)(0)(1)(2\ten)(34)(59)(67)(8)$.
Then, $\CC_{12}$ is stable under the action of 
$\delta \in \FS(\Omega)$.

{\rm (5)} Set $\sigma':=\nu^{-1}\circ\delta$. 
Then, $\sigma'=(\infty)(4)(7)(012)(35\ten)(689)$. 
In particular, the order of $\sigma'$ is equal to $3$. 
\end{thm}

We now set (see \cite[Chapter 18, \S4, I\hspace{-1pt}I]{CS})
\begin{equation*}
(Q \subset) \quad 
\Ni(A_{2}^{12}):=\bigsqcup_{C \in \CC_{12}} C \quad (\subset Q^{\ast}).
\end{equation*}
From now on, we arrange the coordinate of 
$Q^{\ast}=(A_{2}^{\ast})^{12}$ as follows:
\begin{equation*}
(\mu_{i})_{i \in \Omega}=
\bigl(
 \underbrace{\mu_{\infty},\,\mu_{4},\,\mu_{7}}_{\in (A_{2}^{\ast})^{3}} \mid
 \underbrace{\mu_{0},\,\mu_{3},\,\mu_{6}}_{\in (A_{2}^{\ast})^{3}} \mid
 \underbrace{\mu_{1},\,\mu_{5},\,\mu_{8}}_{\in (A_{2}^{\ast})^{3}} \mid
 \underbrace{\mu_{2},\,\mu_{10},\,\mu_{9}}_{\in (A_{2}^{\ast})^{3}}
\bigr).
\end{equation*}
We first define an automorphism $\sigma'$ of $Q^{\ast}$ of order $3$ by:
\begin{equation*}
\bigl(
 \bmu_{\infty47} \mid
 \bmu_{036} \mid
 \bmu_{158} \mid
 \bmu_{2\ten9}
\bigr) 
\stackrel{\sigma'}{\longrightarrow} 
\bigl(
 \bmu_{\infty47} \mid 
 \bmu_{2\ten9} \mid \bmu_{036} \mid \bmu_{158}
\bigr)
\end{equation*}
for $\bmu_{\infty47},\,\bmu_{036},\,
\bmu_{158},\,\bmu_{2\ten9} \in (A_{2}^{\ast})^{3}$.
Then we deduce from Theorem~\ref{thm:Golay}\,(5) that 
$\sigma'$ stabilizes the Niemeier lattice $\Ni(A_{2}^{12})$, 
and hence $\sigma' \in \Aut(\Ni(A_{2}^{12}))$. 
Next, let $\vp=\vp_{A_{2}}$ denote the lattice automorphism of $A_{2}$ 
defined by: $(x_0,x_1,x_2) \mapsto (x_2,x_0,x_1)$; note that 
$\vp(\ol{[\ell]})=\ol{[\ell]}$ for every $\ell=0,\,1,\,2$. 
So, if we define an automorphism $\sigma''$ of $Q^{\ast}$ 
(of order $3$) by:
\begin{equation*}
\bigl(
 \mu_{\infty},\,\mu_{4},\,\mu_{7} \mid \bmu_{036} \mid 
 \bmu_{158} \mid \bmu_{2\ten9}
\bigr) 
\stackrel{\sigma''}{\longrightarrow} 
\bigl(
 \vp(\mu_{\infty}),\,\vp(\mu_{4}),\,\vp(\mu_{7}) \mid \bmu_{036} \mid 
 \bmu_{158} \mid \bmu_{2\ten9}
\bigr)
\end{equation*}
for $\mu_{\infty},\,\mu_{4},\,\mu_{7} \in A_{2}^{\ast}$ and 
$\bmu_{036},\,\bmu_{158},\,\bmu_{2\ten9} \in (A_{2}^{\ast})^{3}$, 
then $\sigma''$ stabilizes the Niemeier lattice $\Ni(A_{2}^{12})$, 
and hence $\sigma'' \in \Aut(\Ni(A_{2}^{12}))$. Set
\begin{equation*}
\sigma_{1}:=\sigma' \circ \sigma'';
\end{equation*}
it is obvious that $\sigma_{1}$ is of order $3$. 
The fixed point subspace $\Fh_{(0)}$ 
of $\Fh=\Ni(A_{2}^{12}) \otimes \BC$ under the $\sigma_{1}$ above 
is identical to
\begin{equation*}
\bigl\{(\bzero \mid \bh \mid \bh \mid \bh) \mid 
\bh \in A_{2}^{3} \otimes \BC=(A_{2} \otimes \BC)^{3} \bigr\}, 
\end{equation*}
where $\bzero$ denotes the zero vector in 
$A_{2}^{3} \otimes \BC=(A_{2} \otimes \BC)^{3}$. 
Thus, $\dim \Fh_{(0)}=2 \times 3 = 6$, and hence
%
%
\begin{equation} \label{eq:rho}
\rho=\frac{1}{18}(\dim \Fh_{(1)} + \dim \Fh_{(2)})=
\frac{1}{18}(24-\dim \Fh_{(0)})=1.
\end{equation}
Also, remark that
%
%
\begin{equation} \label{eq:h12}
\Fh_{(1)} \oplus \Fh_{(2)}=N \otimes_{\BZ} \BC = 
\bigl\{
 (\bh \mid \bh_1 \mid \bh_2 \mid -\bh_1-\bh_2) \mid 
  \bh,\,\bh_1,\,\bh_2 \in A_{2}^{3} \otimes \BC
\bigr\};
\end{equation}
for the definition of $N$, see \eqref{eq:N}. 

By \eqref{eq:rho}, we can apply Theorem~\ref{thm:M} 
to $L=\Ni(A_{2}^{12})$ and the $\sigma_{1}$ above, and 
obtain a $C_{2}$-cofinite, holomorphic VOA $\ti{V}_{L}^{\sigma_{1}}$
of central charge $24$. 
We are ready to state the main result in this section. 
%
%
\begin{thm} \label{thm:main}
Keep the notation and setting above. 
For $L=\Ni(A_{2}^{12})$ and the $\sigma_{1}$ above, 
the Lie algebra $(\ti{V}_{L}^{\sigma_{1}})_{1}$ is 
of type $A_{2,3}^{6}$. 
Therefore, $\ti{V}_{L}^{\sigma_{1}}$ corresponds to 
No.\,6 on Schellekens' list \cite[Table 1]{Sch}.  
\end{thm}

%
\subsection{Proof of Theorem~\ref{thm:main}.}
\label{subsec:prf-main}

First, let us determine the 
Lie algebra structure of $(V_{L}^{\sigma_{1}})_{1}$. 
For $h \in A_{2} \otimes \BC$, define $h^{(i)}$ to be 
$(h_{i})_{i \in \Omega} \in \Fh=L \otimes \BC$ with
$h_{i}=h$ and $h_{j}=0$ for all $j \in \Omega$, $j \ne i$, and set
\begin{align*}
& h^{(012)}:=h^{(0)}+h^{(1)}+h^{(2)}=
  (0,\,0,\,0 \mid h,\,0,\,0 \mid h,\,0,\,0 \mid h,\,0,\,0), \\
& h^{(35\ten)}:=h^{(3)}+h^{(5)}+h^{(10)}=
  (0,\,0,\,0 \mid 0,\,h,\,0 \mid 0,\,h,\,0 \mid 0,\,h,\,0), \\
& h^{(689)}:=h^{(6)}+h^{(8)}+h^{(9)}=
  (0,\,0,\,0 \mid 0,\,0,\,h \mid 0,\,0,\,h\mid 0,\,0,\,h);
\end{align*}
note that $h^{(012)}=h^{(0)}+\sigma_{1} h^{(0)} + \sigma_{1}^{2} h^{(0)}$, and so on. 
Observe that the set $\Delta=\Delta(L)$ 
of roots in $L$ is identical to
$\bigl\{\alpha^{(i)} \mid \alpha \in \Delta(A_{2}), \, 
i \in \Omega\bigr\}$. 
We see that $\Delta$ is stable under $\sigma_{1} \in \Aut(L)$, and 
the action of $\sigma_{1}$ on $\Delta$ is fixed point free. 
Define $\Fg^{(012)}$ to be the subspace spanned by 
$\bigl\{h^{(012)}(-1) \otimes e^{0} \mid h \in A_{2} \otimes \BC \bigr\}$ 
and $\bigl\{1 \otimes e^{\alpha^{(0)}}+1 \otimes e^{\alpha^{(1)}}+
1 \otimes e^{\alpha^{(2)}} 
\mid \alpha \in \Delta(A_{2})\bigr\}$, 
and define $\Fg^{(35X)}$ and $\Fg^{(689)}$ in exactly 
the same way as $\Fg^{(012)}$ with $0,1,2$ replaced by 
$3,5,10$ and $6,8,9$, respectively. Also, define 
\begin{equation*}
\Fa:=\Span_{\BC}
 \bigl\{
   1 \otimes e^{\alpha^{(i)}}+
   1 \otimes e^{\vp(\alpha)^{(i)}}+
   1 \otimes e^{\vp^{2}(\alpha)^{(i)}} \mid 
 \alpha \in \Delta(A_{2}),\,i=\infty,\,4,\,7\bigr\};
\end{equation*}
note that $\dim \Fa = 2 \times 3=6$ 
since $|\Delta(A_{2})/\vp|=2$.
%
%
\begin{lem} \label{lem:fixed}
As a vector space, 
%
%
\begin{equation} \label{eq:fixed}
(V_{L}^{\sigma_{1}})_{1} = 
\Fg^{(012)} \oplus \Fg^{(35X)} \oplus \Fg^{(689)} \oplus \Fa, 
\end{equation}
and hence $\dim (V_{L}^{\sigma_{1}})_{1}=8 \times 3 + 6 = 30$.
Furthermore, $\Fg^{(012)}$, $\Fg^{(35X)}$, and $\Fg^{(689)}$ 
are ideals of the (whole) Lie algebra $(\ti{V}_{L}^{\sigma_{1}})_{1}$ 
isomorphic to the simple Lie algebra of type $A_{2}$. 
In addition, $\Fa$ is an abelian Lie subalgebra, and 
the (adjoint) actions of an element of $\Fa$ on 
$V_{L}(\sigma_{1})_{1}$ and $V_{L}(\sigma_{1}^{2})_{1}$ are semisimple. 
\end{lem}

\begin{proof}
We can easily check \eqref{eq:fixed}. 
Let us show the assertion for $\Fg^{(012)}$; 
we can show the assertions for $\Fg^{(35X)}$ and 
$\Fg^{(689)}$ similarly. 
It follows immediately from the definition of 
the vertex operator for $V_{L}$ 
(see \eqref{eq:Yh} and \eqref{eq:Ye}) that 
$\Fg^{(012)}$ is a Lie subalgebra 
isomorphic to the simple Lie algebra of type $A_{2}$, 
and 
\begin{equation*}
[\Fg^{(012)},\,\Fg^{(35X)}]=
[\Fg^{(012)},\,\Fg^{(689)}]=
[\Fg^{(012)},\,\Fa]=\bigl\{0\}.
\end{equation*}
Let us show that $[a,\,b]=a_{0}b=0$ 
for all $a \in \Fg^{(012)}$ and $b \in V_{L}(\sigma_{1})_{1}$. 
Assume that $a=h^{(012)}(-1) \otimes e^{0}$ 
for some $h \in A_{2} \otimes \BC$. 
Since $V_{L}(\sigma_{1})_{1}$ is the top weight space of 
$V_{L}(\sigma_{1})$ by \eqref{eq:rho}, 
it follows immediately from Lemma~\ref{lem:Ysh}\,(1), 
along with \eqref{eq:Yti}, that $a_{0}b=0$ for all 
$b \in V_{L}(\sigma_{1})_{1}$. 
Assume that 
$a=1 \otimes e^{\alpha^{(0)}}+1 \otimes e^{\alpha^{(1)}}+
1 \otimes e^{\alpha^{(2)}}$ for some $\alpha \in \Delta(A_{2})$. 
Notice that $\alpha^{(i)}$, $i=0,\,1,\,2$, is not contained in $N$ 
by \eqref{eq:h12}.
Thus we see from Lemma~\ref{lem:Ysh}\,(2), 
along with \eqref{eq:Yti}, that $a_{0}b=0$ for all 
$b \in V_{L}(\sigma_{1})_{1}$. 
Similarly, we can show that 
$[a,\,b]=a_{0}b=0$ for all $a \in \Fg^{(012)}$ and 
$b \in V_{L}(\sigma_{1}^{2})_{1}$. Thus we have proved 
that $\Fg^{(012)}$ is an ideal of $(\ti{V}_{L}^{\sigma_{1}})_{1}$. 

Now, it can be easily checked 
by the definition of the vertex operator for $V_{L}$ 
(see \eqref{eq:Ye}) that 
$\Fa$ is an abelian Lie subalgebra. 
Let (and fix) $\alpha \in \Delta(A_{2})$, 
and $i=\infty,\,4,\,7$. Note that $\vp^{r}(\alpha)^{(i)} \in N$ by \eqref{eq:h12}. 
Hence we see from 
Lemma~\ref{lem:Ysh}\,(3) that 
$(1 \otimes e^{\vp^{r}(\alpha)^{(i)}})_{0}$ is 
semisimple on $V(\sigma_{1})_{1}$ for each $r=0,\,1,\,2$. 
Since $\pair{\vp^{r}(\alpha)}{\vp^{r+1}(\alpha)}=-1$ 
for every $r=0,\,1,\,2$, it follows immediately 
from the definition \eqref{eq:c0s} of 
the commutator map $c_{0}^{\sigma_{1}}$ for $\ha{L}_{\sigma_{1}}$ that 
$c_{0}^{\sigma_{1}}\bigl(\vp^{r}(\alpha)^{(i)},\,\vp^{r+1}(\alpha)^{(i)}\bigr)=0$
for every $r=0,\,1,\,2$. 
Thus, $(\kappa^{0},\,e_{ \vp^{r}(\alpha)^{(i)} }) \in \ha{L}_{\sigma_{1}}$, 
$r=0,\,1,\,2$, commute with each other, and hence 
so are $(1 \otimes e^{\vp^{r}(\alpha)^{(i)}})_{0} \in 
\End_{\BC}(V(\sigma_{1})_{1})$, $r=0,\,1,\,2$, 
because $(1 \otimes e^{\vp^{r}(\alpha)^{(i)}})_{0}$ is identical to a scalar multiple of 
the action of $(\kappa^{0},\,e_{ \vp^{r}(\alpha)^{(i)} })$
on $V(\sigma_{1})_{1}$ (see the proof of Lemma~\ref{lem:Ysh}\,(3)). 
Therefore we conclude that $(1 \otimes e^{\alpha^{(i)}})_{0}+
(1 \otimes e^{\vp(\alpha)^{(i)}})_{0}+
(1 \otimes e^{\vp^{2}(\alpha)^{(i)}})_{0}$ 
is also semisimple on $V(\sigma_{1})_{1}$. Similarly, 
we can show that $(1 \otimes e^{\alpha^{(i)}})_{0}+
(1 \otimes e^{\vp(\alpha)^{(i)}})_{0}+
(1 \otimes e^{\vp^{2}(\alpha)^{(i)}})_{0}$ 
is semisimple also on $V(\sigma_{1}^{2})_{1}$.
This completes the proof of the lemma. 
\end{proof}

We see from Proposition~\ref{prop:DM} and Lemma~\ref{lem:fixed} 
that the Lie algebra $(\ti{V}_{L}^{\sigma_{1}})_{1}$ is semisimple. 
Also, we deduce from the definition that the levels of 
the simple components $\Fg^{(012)}$, $\Fg^{(35X)}$, and 
$\Fg^{(689)}$ of $(\ti{V}_{L}^{\sigma_{1}})_{1}$ are 
all equal to $3$. Since the dual Coxeter number of 
$A_{2}$ is equal to $3$, it follows immediately 
from \eqref{eq:level} that
\begin{equation*}
\frac{3}{3}=\frac{\dim (\ti{V}_{L}^{\sigma_{1}})_{1} - 24}{24}, \quad 
\text{and hence} \quad \dim (\ti{V}_{L}^{\sigma_{1}})_{1} = 48.
\end{equation*}
Therefore we have $\dim V(\sigma_{1})_{1}=\dim V(\sigma_{1}^{2})_{1}=9$ 
(see Remark~\ref{rem:dual}). 

\begin{rem}
We see from \eqref{eq:top} that $\dim V(\sigma_{1})_{1}=|N/R|^{1/2}$. 
Also it can be shown that $R=M=(1-\sigma_{1})L$. Using these facts, 
we can determine the dimension of $V(\sigma_{1})_{1}$ 
also by lattice theoretic method. 
\end{rem}

Set 
\begin{equation*}
\Fg:=\Fa \oplus 
 \underbrace{V(\sigma_{1})_{1}}_{=:\Fg_{1}} \oplus 
 \underbrace{V(\sigma_{1}^{2})_{1}}_{=:\Fg_{2}} 
\subset (\ti{V}_{L}^{\sigma_{1}})_{1}. 
\end{equation*}
By Remark~\ref{rem:Z3} and the argument above, 
we see that $\Fg$ is an ideal of the semisimple Lie algebra 
$(\ti{V}_{L}^{\sigma_{1}})_{1}$ satisfying the following conditions:

(i) $\Fg$ is a $(\BZ/3\BZ)$-graded, semisimple Lie algebra of dimension $24$;

(ii) $\Fa$ is an abelian subalgebra of dimension 6 whose (adjoint) actions on 
$\Fg_{1}$ and $\Fg_{2}$ are both semisimple.

%
\begin{prop} \label{prop:A2-3}
The Lie algebra $\Fg$ above is isomorphic to 
the semisimple Lie algebra of type $A_{2}^{3}$.
\end{prop}

\begin{proof}
We see by (ii) and \cite[Lemma 8.1\,b)]{Kac} 
that the centralizer $\Fz=\bigl\{x \in \Fg \mid 
[x,\,h]=0 \text{ for all $h \in \Fa$}\bigr\}$ of $\Fa$ 
is a Cartan subalgebra of $\Fg$. Notice that 
\begin{equation*}
\Fz=\Fa \oplus (\Fz \cap \Fg_{1}) \oplus (\Fz \cap \Fg_{2}).
\end{equation*}
Suppose that $\Fz \cap \Fg_{1} \ne \bigl\{0\bigr\}$. 
Let $h \in \Fz \cap \Fg_{1}$, and let $\alpha$ be a root of $\Fg$
such that $\alpha(h) \ne 0$. Take a nonzero root vector $x \in \Fg$ 
with respect to $\alpha$, and 
write it as: $x=x_{0}+x_{1}+x_{2}$ with $x_{0} \in \Fa$ and 
$x_{i} \in \Fg_{i}$, $i=1,\,2$; 
note that $[h,\,x_{0}]=0$. Then, 
\begin{equation*}
\alpha(h)(x_{0}+x_{1}+x_{2}) = \alpha(h)x = 
[h,\,x] = [h,\,x_{0}+x_{1}+x_{2}]=
 \underbrace{[h,\,x_{0}]}_{=0} +  
 \underbrace{[h,\,x_{1}]}_{\in \Fg_{2}} + 
 \underbrace{[h,\,x_{2}]}_{\in \Fa}.
\end{equation*}
Since $\alpha(h) \ne 0$, we get $x_{1}=0$. Also, 
since $0=[h,\,x_{1}]=\alpha(h)x_{2}$, and $\alpha(h) \ne 0$, 
we have $x_{2}=0$. Similarly, $x_{0}=0$. Thus we obtain $x=0$, 
which is a contradiction. Hence, 
$\Fz \cap \Fg_{1} = \bigl\{0\bigr\}$. Similarly, 
we can show that $\Fz \cap \Fg_{2} = \bigl\{0\bigr\}$. 
Therefore we conclude that $\Fa$ is a Cartan subalgebra of $\Fg$.  

Since the rank of $\Fg$ is equal to $\dim \Fa=6$, 
we deduce from (i) that 
$\Fg$ is isomorphic to the semisimple Lie algebra of type
$A_{2}^{3}$ or $B_{2}A_{2}A_{1}^{2}$.
Suppose that $\Fg$ is of type 
$B_{2}A_{2}A_{1}^{2}$.
Since $\Fg_{1}$ and $\Fg_{2}$ are stable under 
the adjoint action of the Cartan subalgebra $\Fa$, 
it follows that $\Fg_{1}$ and $\Fg_{2}$ are direct sums of 
root spaces of $\Fg$ with respect to the Cartan subalgebra $\Fa$, 
and each root space of $\Fg$ is contained in either $\Fg_{1}$ or $\Fg_{2}$. 
Let $\Fs$ be the ideal of $\Fg$ isomorphic to 
the simple Lie algebra $\Fg_{B_{2}}$ of type $B_{2}$, and 
let $\alpha_{1},\,\alpha_{2}$ be the simple roots 
for $\Fs \cong \Fg_{B_{2}}$ such that 
$\theta:=2\alpha_{1}+\alpha_{2}$ is the highest root 
of $\Fs \cong \Fg_{B_{2}}$. 
We suppose that $\Fs_{\theta} \subset \Fg_{1}$ 
(resp., $\Fs_{\theta} \subset \Fg_{2}$). 
If the root space $\Fs_{-\alpha_{1}}$ is 
contained in $\Fg_{2}$ (resp., $\Fg_{1}$), then we have
\begin{equation*}
\{0\} \ne \Fs_{\alpha_{1}+\alpha_{2}}=
 [\Fs_{\theta},\,\Fs_{-\alpha_{1}}] \subset \Fa
\end{equation*}
by (i), which is a contradiction. 
If $\Fs_{-\alpha_{1}} \subset \Fg_{1}$ 
(resp., $\Fs_{-\alpha_{1}} \subset \Fg_{2}$), 
then we have
\begin{equation*}
\{0\} \ne \Fs_{\alpha_{2}}=
 [[\Fs_{\theta},\,\Fs_{-\alpha_{1}}],\,\Fs_{-\alpha_{1}}] \subset \Fa
\end{equation*}
by (i), which is also a contradiction. 
Thus we have proved that 
$\Fg$ is not of type $B_{2}A_{2}A_{1}^{2}$, 
thereby completing the proof of the proposition. 
\end{proof}

Theorem~\ref{thm:main} follows immediately 
from Lemma~\ref{lem:fixed}, Proposition~\ref{prop:A2-3}, 
and \eqref{eq:level}. 

%
\section{Construction of a holomorphic VOA (2).}
\label{sec:hvoa2}

%
\subsection{Root lattice $D_{4}$ and its fixed-point-free automorphism of order $3$.} 
\label{subsec:D4}

Following \cite[Chapter~4, \S7.1]{CS}, 
we set
\begin{equation*}
D_{4}:=\bigl\{(x_{1},\,x_{2},\,x_{3},\,x_{4}) \in \BZ^{4} \mid 
x_{1}+x_{2}+x_{3}+x_{4} \in 2\BZ\bigr\};
\end{equation*}
\begin{equation*}
\Delta(D_{4}):=\bigl\{\alpha \in D_{4} \mid 
\pair{\alpha}{\alpha}=2\bigr\} \quad 
\text{(the set of roots in $D_{4}$)},
\end{equation*}
\begin{align*}
& [0]:=(0,\,0,\,0,\,0) \in D_{4}^{\ast}, & 
& [1]:=(1/2,\,1/2,\,1/2,\,1/2) \in D_{4}^{\ast}, \\
& [2]:=(0,\,0,\,0,\,1) \in D_{4}^{\ast}, & 
& [3]:=(1/2,\,1/2,\,1/2,\,-1/2) \in D_{4}^{\ast},
\end{align*}
where $D_{4}^{\ast} \subset D_{4} \otimes_{\BZ} \BR$ 
denotes the dual lattice of $D_{4}$. 
If we set $\ol{[\ell]}:=[\ell]+D_{4}$ for $\ell=0,\,1,\,2,\,3$, 
then we have 
$D_{4}^{\ast}/D_{4}=\bigl\{\ol{[\ell]} \mid \ell=0,\,1,\,2,\,3\bigr\}$, 
and the additive group $D_{4}^{\ast}/D_{4}$ is naturally isomorphic to 
$\BZ/2\BZ \times \BZ/2\BZ$; in fact, we have $\ol{[\ell]}+\ol{[\ell]}=\ol{[0]}$ 
for all $\ell=0,\,1,\,2,\,3$, and $\ol{[1]}+\ol{[2]}=\ol{[3]}$. 

We define a linear automorphism $\vp=\vp_{D_{4}}$ of $D_{4} \otimes \BR$ by: 
\begin{align*}
& (1,0,0,0) \mapsto \frac{1}{2}(-1,\,1,\,1,\,1), &
& (0,1,0,0) \mapsto \frac{1}{2}(-1,\,-1,\,1,\,-1), \\[1mm]
& (0,0,1,0) \mapsto \frac{1}{2}(-1,\,-1,\,-1,\,1), & 
& (0,0,0,1) \mapsto \frac{1}{2}(-1,\,1,\,-1,\,-1);
\end{align*}
we deduce that the restriction of $\vp$ to $D_{4}$ (resp., $D_{4}^{\ast}$) 
is a lattice automorphism of order $3$ of $D_{4}$ (resp., $D_{4}^{\ast}$), 
which is fixed point free on $D_{4}$ (resp., $D_{4}^{\ast}$). 
Remark that 
%
%
\begin{equation} \label{eq:vpD4}
\vp(\ol{[0]})=\ol{[0]}, \quad
\vp(\ol{[1]})=\ol{[2]}, \quad
\vp(\ol{[2]})=\ol{[3]}, \quad
\vp(\ol{[3]})=\ol{[1]}.
\end{equation}

%
\subsection{Niemeier lattice $\Ni(D_{4}^{6})$ and 
its automorphism $\sigma_{2}$ of order $3$.}
\label{subsec:Nie2}

The Niemeier lattice $L=\Ni(D_{4}^{6})$ with $Q=D_{4}^{6}$ the root lattice 
is, by definition (see \cite[Chapter 16, Table 16.1]{CS}), 
the sublattice of $Q^{\ast}=(D_{4}^{\ast})^{6}$ generated by 
$Q$, $[111111]$, $[222222]$,%
\footnote{There seems to be a typo in the row of $D_{4}^{6}$ on 
\cite[Chapter 16, Table 16.1]{CS}; $[222222]$ should be added to the column 
``Generators for glue code''.} and 
\begin{equation*}
[0(02332)]=
\bigl\{[002332],\,[023320],\,[033202],\,[032023],\,[020233]\bigr\},
\end{equation*}
where $[a_{1} \cdots a_{6}]:=([a_{1}],\,\dots,\,[a_{6}]) \in 
Q^{\ast}=(D_{4}^{\ast})^{6}$.

Let us define $\sigma_{2}:Q^{\ast} \rightarrow Q^{\ast}$ by: 
$\sigma_{2}=\vp^{\oplus 6}$, 
that is, $\sigma_{2}(\gamma_{1},\,\dots,\,\gamma_{6})=
(\vp(\gamma_{1}),\,\dots,\,\vp(\gamma_{6}))$. 
Then, $L \subset Q^{\ast}$ is stable under the action of $\sigma_{2}$. 
Indeed, it is obvious that $\sigma_{2}(Q) \subset Q$. 
Also, $\sigma_{2}([111111]) \in [222222]+Q \subset L$, and 
$\sigma_{2}([222222]) \in [333333]+Q = [111111]+[222222]+Q \subset L$. 
In addition,
\begin{equation*}
\sigma_{2}([002332]) \in [003113]+Q = 
 [020233] + [023320] + Q \subset L. 
\end{equation*}
Similarly, we can show that $\sigma_{2}([0(02332)]) \subset L$. 
Therefore, $\sigma_{2}$ is a lattice automorphism of $L$ of order $3$. 
Because $\vp$ is fixed point free on $D_{4}^{\ast}$, it follows that 
$\sigma_{2}$ is fixed point free on $L$, which implies that
$\rank L^{\sigma_{2}}=0$, and hence 
%
%
\begin{equation} \label{eq:rho2}
\rho=\frac{1}{18}(\dim \Fh_{(1)} + \dim \Fh_{(2)})=
\frac{1}{18}(24-0)=\frac{4}{3}. 
\end{equation}
Thus we can apply Theorem~\ref{thm:M} 
to $L=\Ni(D_{4}^{6})$ and the $\sigma_{2}$ above, and 
obtain a $C_{2}$-cofinite, holomorphic VOA $\ti{V}_{L}^{\sigma_{2}}$
of central charge $24$; note that 
$\dim (\ti{V}_{L}^{\sigma_{2}})_{0}=\dim (V_{L}^{\sigma_{2}})_{0}=1$. 
We are ready to state the main result in this section. 
%
%
\begin{thm} \label{thm:main2}
Keep the notation and setting above. 
For $L=\Ni(D_{4}^{6})$ and the $\sigma_{2}$ above, 
the Lie algebra $(\ti{V}_{L}^{\sigma_{2}})_{1}$ is of type $A_{2,3}^{6}$. 
Therefore, $\ti{V}_{L}^{\sigma_{2}}$ corresponds to 
No.\,6 on Schellekens' list \cite[Table 1]{Sch}.  
\end{thm}

\begin{proof}
By \eqref{eq:rho2} and \eqref{eq:wt}, we see that 
both of the top weights of $V(\sigma_{2})$ and $V(\sigma_{2}^{2})$ 
are equal to $4/3$. Hence, 
$(\ti{V}_{L}^{\sigma_{2}})_{1}=(V_{L}^{\sigma_{2}})_{1}$. 
So, let us determine the Lie algebra structure of 
$(V_{L}^{\sigma_{2}})_{1}$. For each $h \in D_{4} \otimes_{\BZ} \BC$ 
and $1 \le i \le 6$, we define $h^{(i)}$  to be the element 
$(h_{1},\,\dots,\,h_{6}) \in \Fh=L \otimes_{\BZ} \BC = 
(D_{4} \otimes_{\BZ} \BC)^{6}$ with $h_{i}=h$ and $h_{j}=0$ 
for all $1 \le j \le 6$, $j \ne i$. Define $\Fg^{(i)}$ to be 
the Lie subalgebra of $(V_{L})_{1}$ generated by 
$\bigl\{h^{(i)}(-1)1 \otimes e^{0} \mid h \in D_{4} \otimes_{\BZ} \BC\bigr\}$ 
and $\bigl\{1 \otimes e^{\alpha^{(i)}} \mid \alpha \in \Delta(D_{4})\bigr\}$. 
Then we have 
\begin{equation*}
(V_{L})_{1} \cong 
 \Fg^{(1)} \oplus \cdots \oplus \Fg^{(6)} \quad \text{with} \quad 
 \Fg^{(i)} \cong \Fg_{D_{4}} \quad \text{for every $1 \le i \le 6$}.
\end{equation*}
It can be easily seen that $\sigma_{2} \in \Aut(V_{L})$ preserves 
each simple component $\Fg^{(i)}$ ($\cong \Fg_{D_{4}}$) of $(V_{L})_{1}$; 
indeed, for each $1 \le i \le 6$, we see that 
\begin{equation*}
\begin{cases}
\sigma_{2}(h^{(i)}(-1)1 \otimes e^{0})=
(\vp(h))^{(i)}(-1)1 \otimes e^{0} & 
 \text{for $h \in D_{4} \otimes_{\BZ} \BC$}; \\[1.5mm]
\sigma_{2}(1 \otimes e^{\alpha^{(i)}})=1 \otimes e^{\vp(\alpha)^{(i)}} & 
 \text{for $\alpha \in \Delta(D_{4})$}.
\end{cases}
\end{equation*}
Thus the restriction of $\sigma_{2}$ to a simple component 
$\Fg^{(i)} \cong \Fg_{D_{4}}$ is identical to 
the Lie algebra automorphism, denoted also by $\vp$, of $\Fg_{D_{4}}$ 
induced from $\vp \in \Aut(D_{4})$. Therefore, 
\begin{equation*}
(\ti{V}_{L}^{\sigma_{2}})_{1} = 
(V_{L}^{\sigma_{2}})_{1} = 
(\Fg^{(1)})^{\sigma_{2}} \oplus \cdots \oplus (\Fg^{(6)})^{\sigma_{2}} \cong 
(\Fg_{D_{4}}^{\vp})^{\oplus 6}.
\end{equation*}
Because $\vp$ is fixed point free on the root lattice $D_{4}$, 
we see that $\dim \Fg_{D_{4}}^{\vp} = 0 + 24/3=8$ 
(recall that the number of roots of $D_{4}$ is equal to $24$). 
Thus we have $\dim (\ti{V}_{L}^{\sigma_{2}})_{1} = 8 \times 6 = 48$. 
Therefore it follows from Proposition~\ref{prop:DM} 
that $(\ti{V}_{L}^{\sigma_{2}})_{1}  \cong 
(\Fg_{D_{4}}^{\vp})^{\oplus 6}$ is semisimple, and hence 
so is $\Fg_{D_{4}}^{\vp}$. 
Observe that $\Fg_{A_{2}}$ is 
a unique semisimple Lie algebra of dimension $8$. 
Thus, $\Fg_{D_{4}}^{\vp}$ is of type $A_{2}$, 
and hence $(\ti{V}_{L}^{\sigma_{2}})_{1}$ is of type $A_{2}^{6}$. 
Since $\dim (\ti{V}_{L}^{\sigma_{2}})_{1} = 48$ as seen above, 
we see from \eqref{eq:level} that 
$(\ti{V}_{L}^{\sigma_{2}})_{1}$ is of type $A_{2,3}^{6}$, as desired. 
\end{proof}

%
\section{Construction of a holomorphic VOA (3).}
\label{sec:hvoa3}

%
\subsection{Dynkin diagram automorphism for $D_{4}$.}
\label{subsec:DynkinD4}

We use the notation and setting 
for the root lattice $D_{4}$ introduced in \S\ref{subsec:D4}.
Let $\omega$ be a Dynkin diagram automorphism of $D_{4}$ of order $3$. 
Then, $\omega$ acts on $D_{4}$ as follows: 
\begin{align*}
& \omega(1,\,-1,\,0,\,0)=(0,\,0,\,1,\,-1), & 
& \omega(0,\,0,\,1,\,-1)=(0,\,0,\,1,\,1), \\
& \omega(0,\,0,\,1,\,1)=(1,\,-1,\,0,\,0), & 
& \omega(0,\,1,\,-1,\,0)=(0,\,1,\,-1,\,0);
\end{align*}
we can easily check that 
\begin{equation*}
\omega(\ol{[0]})=\ol{[0]}, \quad
\omega(\ol{[1]})=\ol{[2]}, \quad
\omega(\ol{[2]})=\ol{[3]}, \quad
\omega(\ol{[3]})=\ol{[1]}.
\end{equation*}

%
\subsection{Niemeier lattice $\Ni(D_{4}^{6})$ and 
its automorphism $\sigma_{3}$ of order $3$.}
\label{subsec:Nie3}

Keep the notation and setting for 
the Niemeier lattice $L=\Ni(D_{4}^{6})$ 
with $Q=D_{4}^{6}$ the root lattice in \S\ref{subsec:Nie2}.
Let us define $\sigma_{3}:Q^{\ast} \rightarrow Q^{\ast}$ by: 
\begin{equation*}
\sigma_{3}(
 \gamma_{1},\,\gamma_{2},\,\gamma_{3},\,
 \gamma_{4},\,\gamma_{5},\,\gamma_{6})=
(
 \vp(\gamma_{1}),\,\vp(\gamma_{2}),\,\vp(\gamma_{3}),\,
 \omega(\gamma_{4}),\,\omega(\gamma_{5}),\,\omega(\gamma_{6})).
\end{equation*}
We can show in exactly the same way as for $\sigma_{2}$ 
in \S\ref{subsec:Nie2} that $L \subset Q^{\ast}$ is 
stable under the action of $\sigma_{3}$. 
Therefore, $\sigma_{3}$ gives a lattice automorphism of $L$ of order $3$; 
since $\rank D_{4}^{\vp} = 0$ and $\rank D_{4}^{\omega} = 2$, it follows that 
$\rank L^{\sigma_{3}}=2 \times 3=6$, and hence 
%
%
\begin{equation} \label{eq:rho3}
\rho=\frac{1}{18}(\dim \Fh_{(1)} + \dim \Fh_{(2)})=
\frac{1}{18}(24-6)=1. 
\end{equation}
Thus we can apply Theorem~\ref{thm:M} 
to $L=\Ni(D_{4}^{6})$ and the $\sigma_{3}$ above, and 
obtain a $C_{2}$-cofinite, holomorphic VOA $\ti{V}_{L}^{\sigma_{3}}$
of central charge $24$; note that 
$\dim (\ti{V}_{L}^{\sigma_{3}})_{0}=\dim (V_{L}^{\sigma_{3}})_{0}=1$. 
We are ready to state the main result in this section. 
%
%
\begin{thm} \label{thm:main3}
Keep the notation and setting above. 
For $L=\Ni(D_{4}^{6})$ and the $\sigma_{3}$ above, 
the Lie algebra $(\ti{V}_{L}^{\sigma_{3}})_{1}$ 
is of type $E_{6,3}G_{2,1}^{3}$. 
Therefore, $\ti{V}_{L}^{\sigma_{3}}$ corresponds to 
No.\,32 on Schellekens' list \cite[Table 1]{Sch}.  
\end{thm}

\begin{proof}[Proof of Theorem~\ref{thm:main3}]
First, let us determine the Lie algebra structure of 
$(V_{L}^{\sigma_{3}})_{1}$. 
We use the notation in the proof of Theorem~\ref{thm:main2}; 
recall that 
\begin{equation*}
(V_{L})_{1} \cong 
 \Fg^{(1)} \oplus \cdots \oplus \Fg^{(6)} \quad \text{with} \quad 
 \Fg^{(i)} \cong \Fg_{D_{4}} \quad \text{for every $1 \le i \le 6$}. 
\end{equation*}
It can be easily seen that $\sigma_{3} \in \Aut(V_{L})$ preserves 
each simple component $\Fg^{(i)}$ ($\cong \Fg_{D_{4}}$) of $(V_{L})_{1}$; 
indeed, for each $i=1,\,2,\,3$, 
\begin{equation*}
\begin{cases}
\sigma_{3}(h^{(i)}(-1)1 \otimes e^{0})=
(\vp(h))^{(i)}(-1)1 \otimes e^{0} & 
 \text{for $h \in D_{4} \otimes_{\BZ} \BC$}; \\[1.5mm]
\sigma_{3}(1 \otimes e^{\alpha^{(i)}})=1 \otimes e^{\vp(\alpha)^{(i)}} & 
 \text{for $\alpha \in \Delta(D_{4})$}, 
\end{cases}
\end{equation*}
and for each $i=4,\,5,\,6$, 
\begin{equation*}
\begin{cases}
\sigma_{3}(h^{(i)}(-1)1 \otimes e^{0})=
(\omega(h))^{(i)}(-1)1 \otimes e^{0} & 
 \text{for $h \in D_{4} \otimes_{\BZ} \BC$}; \\[1.5mm]
\sigma_{3}(1 \otimes e^{\alpha^{(i)}})=1 \otimes e^{\omega(\alpha)^{(i)}} & 
 \text{for $\alpha \in \Delta(D_{4})$}.
\end{cases}
\end{equation*}
Thus the restriction of $\sigma_{3}$ to a simple component 
$\Fg^{(i)} \cong \Fg_{D_{4}}$ for $i=1,\,2,\,3$ (resp., for $i=4,\,5,\,6$) 
is identical to the Lie algebra automorphism, denoted also by $\vp$ (resp., $\omega$), 
of $\Fg_{D_{4}}$ induced from $\vp \in \Aut(D_{4})$ 
(resp., the Dynkin diagram automorphism $\omega$). 
Therefore, 
\begin{equation*}
(V_{L}^{\sigma_{3}})_{1} = 
(\Fg^{(1)})^{\sigma_{3}} \oplus \cdots \oplus (\Fg^{(6)})^{\sigma_{3}} \cong 
(\Fg_{D_{4}}^{\vp})^{\oplus 3} \oplus (\Fg_{D_{4}}^{\omega})^{\oplus 3}.
\end{equation*}
We see from the proof of Theorem~\ref{thm:main2} that
$\Fg_{D_{4}}^{\vp}$ is isomorphic to 
the simple Lie algebra $\Fg_{A_{2}}$ of type $A_{2}$. 
Also, it is well-known (see, e.g., \cite[p.128, Case 4]{Kac}) 
that $\Fg_{D_{4}}^{\omega}$ is isomorphic to 
the simple Lie algebra $\Fg_{G_{2}}$ of type $G_{2}$. 

Now, fix $i=4,\,5,\,6$. The Lie algebra 
$(\Fg^{(i)})^{\sigma_{3}} \cong \Fg_{D_{4}}^{\omega} \cong \Fg_{G_{2}}$ 
is generated by: 
\begin{align*}
& \bigl\{h^{(i)}(-1)1 \otimes e^{0} \mid 
    h \in D_{4} \otimes \BC,\,\omega(h)=h\bigr\}, \\
& \bigl\{
    1 \otimes e^{\alpha^{(i)}} + 
    1 \otimes e^{\omega(\alpha)^{(i)}} + 
    1 \otimes e^{\omega^{2}(\alpha)^{(i)}} \mid 
  \alpha \in \Delta(D_{4}),\,\omega(\alpha) \ne \alpha\bigr\}, 
  \quad \text{and} \\
& \bigl\{
    1 \otimes e^{\alpha^{(i)}} \mid 
  \alpha \in \Delta(D_{4}), \omega(\alpha) = \alpha\bigr\}.
\end{align*}
Because $\alpha^{(i)} \in \Delta \setminus N$ for all 
$\alpha \in \Delta(D_{4})$, it follows from Lemma~\ref{lem:Ysh}\,(1) and (2), 
along with \eqref{eq:Yti}, that 
$[a,\,b]=a_{0}b=0$ for all $a \in (\Fg^{(i)})^{\sigma_{3}}$ and 
$b \in V(\sigma_{3})_{1}$. Similarly, we can check that 
$[a,\,b]=a_{0}b=0$ for all $a \in (\Fg^{(i)})^{\sigma_{3}}$ and 
$b \in V(\sigma_{3}^{2})_{1}$. Thus, $(\Fg^{(i)})^{\sigma_{3}} \cong \Fg_{G_{2}}$ 
is a (simple) ideal of $(\ti{V}_{L}^{\sigma_{3}})_{1}$. 
Because $\bigl\{
    1 \otimes e^{\alpha^{(i)}} \mid 
  \alpha \in \Delta(D_{4}), \omega(\alpha) = \alpha\bigr\}$ correspond to 
the root vectors for long roots of $G_{2}$, we deduce that 
the level of $(\Fg^{(i)})^{\sigma_{3}}$ is equal to $1$. Therefore, 
$(\Fg^{(i)})^{\sigma_{3}}$ is of type $G_{2,1}$. Since the dual Coxeter number of 
$G_{2}$ is equal to $4$, we see by \eqref{eq:level} that 
\begin{equation*}
\frac{4}{1}=
\frac{\dim (\ti{V}_{L}^{\sigma_{3}})_{1} -24}{24}, \quad \text{and hence} \quad
\dim (\ti{V}_{L}^{\sigma_{3}})_{1} = 120. 
\end{equation*}

Set
%
%
\begin{equation} \label{eq:Fg}
\Fg:=
\bigl\{
\underbrace{
(\Fg^{(1)})^{\sigma_{3}} \oplus 
(\Fg^{(2)})^{\sigma_{3}} \oplus
(\Fg^{(3)})^{\sigma_{3}}
}_{=:\Fg_{0} \cong \Fg_{A_{2}}^{\oplus 3}}
\bigr\}
\oplus 
\underbrace{V(\sigma_{3})_{1}}_{=:\Fg_{1}} \oplus 
\underbrace{V(\sigma_{3}^{2})_{1}}_{=:\Fg_{2}}
\subset (\ti{V}_{L}^{\sigma_{3}})_{1}.
\end{equation}
By Remark~\ref{rem:Z3}, \eqref{eq:level}, and the argument above, 
we see that $\Fg$ is an ideal of $(\ti{V}_{L}^{\sigma_{3}})_{1}$ 
satisfying the following conditions: 

(i) $\Fg$ is a $(\BZ/3\BZ)$-graded, semisimple Lie algebra 
of dimension $78$; 

(ii) the dual Coxeter number of a simple component of 
$\Fg$ is contained in $4\BZ$.

\noindent
Hence, $\Fg$ is of type
%
%
\begin{equation} \label{eq:e6}
E_{6}, \quad A_{7}A_{3}, \quad 
\text{or} \quad C_{3}^{3}A_{3}.
\end{equation}
We will show that $\Fg$ is of type $E_{6}$. 

Define a linear automorphism $f : \Fg \rightarrow \Fg$ by: 
$f(x)=\zeta^{k}x$ for $x \in \Fg_{k}$, $k=0,\,1,\,2$. 
Then, $f$ is a Lie algebra automorphism of $\Fg$. 
It follows from \cite[Proposition 8.1]{Kac} that
there exist an element $h$ of a Cartan subalgebra $\Fz$ of $\Fg$ 
and a Dynkin diagram automorphism $\kappa$ of $\Fg$ 
preserving $\Fz$ such that $f \in \Aut(\Fg)$ is conjugate to
\begin{equation*}
f':=\kappa \exp \left( \ad \frac{2\pi\sqrt{-1}}{3}h \right).
\end{equation*}
(Observe that \cite[Proposition 8.1]{Kac} is valid for a semisimple Lie algebra. 
In this case, a Dynkin diagram automorphism may permute some isomorphic components.)
Because $f$ is of order $3$, and hence so is $f'$, 
it follows immediately that 
$\kappa^{3}$ is equal to the identity map. 
We first claim that
%
%
\begin{claim} \label{c:kappa}
The Dynkin diagram automorphism $\kappa$ 
is equal to the identity map. 
\end{claim}

\noindent
{\it Proof of Claim~\ref{c:kappa}. }
Suppose that $\kappa$ is not equal to the identity map. 
Then, $\kappa$ is of order $3$. 
Because neither of $E_{6}$ nor $A_{7}A_{3}$ has 
a Dynkin diagram automorphism of order $3$, 
$\Fg$ should be of type $C_{3}^{3}A_{3}$, and $\kappa$ should be 
the Dynkin diagram automorphism of the Dynkin diagram of type $C_{3}^{3}A_{3}$ 
that permutes the (three) $C_{3}$-components, and 
acts on the $A_{3}$-component trivially. 
Write $\Fg$ as: $\Fg=\Fg^{1} \oplus \Fg^{2} \oplus \Fg^{3} \oplus \Fg^{4}$, 
where $\Fg^{1} \cong \Fg^{2} \cong \Fg^{3} \cong \Fg_{C_{3}}$, and 
$\Fg^{4} \cong \Fg_{A_{3}}$; 
we may assume that $\kappa$ maps $(x_{1},\,x_{2},\,x_{3},\,x_{4}) \in \Fg$ 
to $(x_{2},\,x_{3},\,x_{1},\,x_{4}) \in \Fg$. Following this decomposition, 
write the $h$ as: $h=(h_{1},\,h_{2},\,h_{3},\,h_{4})$.  Then we see that
\begin{equation*}
f'(x_{1},\,x_{2},\,x_{3},\,x_{4})=
(e_{2}x_{2},\,e_{3}x_{3},\,e_{1}x_{1},\,e_{4}x_{4}),
\end{equation*}
where for simplicity of notation, we set
$e_{j}:=\exp \left( \ad \frac{2\pi\sqrt{-1}}{3}h_{j} \right) \in \Aut(\Fg_{j})$ 
for $j=1,\,2,\,3,\,4$. Define $g \in \Aut (\Fg)$ by:
$g(x_{1},\,x_{2},\,x_{3},\,x_{4})=
(x_{1},\,e_{2}x_{2},\,e_{1}^{-1}x_{3},\,x_{4})$.
It follows by direct computation that 
\begin{equation*}
gf'g^{-1}(x_{1},\,x_{2},\,x_{3},\,x_{4})=
(x_{2},\,x_{3},\,x_{1},\,e_{4}x_{4}).
\end{equation*}
We deduce that the fixed point Lie subalgebra of $\Fg$ under 
$gf'g^{-1} \in \Aut(\Fg)$ contains an ideal isomorphic to 
the simple Lie algebra of type $C_{3}$. Because $f \in \Aut(\Fg)$ is 
conjugate to $f'$, and hence to $gf'g^{-1}$, it follows that 
the fixed point Lie subalgebra $\Fg_{0}$ of $\Fg$ under 
$f$ also contains an ideal isomorphic to 
the simple Lie algebra of type $C_{3}$. However, 
this is a contradiction, since $\Fg_{0}$ is of type $A_{2}^{3}$ 
by \eqref{eq:Fg}. Thus we have proved Claim~\ref{c:kappa}. \bqed

\vspace{3mm}

Because $\kappa$ is equal to the identity map, 
we see that the fixed point Lie subalgebra of $\Fg$ under 
$f' \in \Aut(\Fg)$ contains the Cartan subalgebra $\Fz$ of $\Fg$. 
Because $f \in \Aut(\Fg)$ is conjugate to $f'$, 
the fixed point Lie subalgebra $\Fg_{0}$ of $\Fg$ under 
$f$ also contains a Cartan subalgebra of $\Fg$. 
Thus the rank of $\Fg$ is less than or equal to the rank of $\Fg_{0}$, 
which is equal to $6$. Thus we conclude by \eqref{eq:e6} that 
$\Fg$ is of type $E_{6}$. 
This completes the proof of the theorem. 
\end{proof}

%
\section{Construction of a holomorphic VOA (4).}
\label{sec:hvoa4}

%
\subsection{Automorphism of order $3$ on the root lattice $D_{4}$.}
\label{subsec:D4psi}

We use the notation for the root lattice $D_{4}$ 
introduced in \S\ref{subsec:D4}. We define 
a lattice automorphism $\psi=\psi_{D_{4}} \in \Aut(D_{4})$ by: 
\begin{equation*}
(x_{1},\,x_{2},\,x_{3},\,x_{4}) \mapsto 
(x_{2},\,x_{3},\,x_{1},\,x_{4}).
\end{equation*}
Observe that $\psi$ acts on the set $\Delta(D_{4})$ of roots 
fixed-point-freely. Set
\begin{align*}
& \beta_{1} := (1,\,-1,\,0,\,0) \in \Delta(D_{4}), & 
& \beta_{2} := (1,\,0,\,0,\,-1) \in \Delta(D_{4}), \\ 
& \beta_{3} := (1,\,1,\,0,\,0) \in \Delta(D_{4}), & 
& \beta_{4} := (1,\,0,\,0,\,1)  \in \Delta(D_{4}).
\end{align*}
Then, $\bigl\{\pm \beta_{j} \mid 1 \le j \le 4\bigr\}$ 
is a complete set of representatives of $\psi$-orbits 
in $\Delta(D_{4})$. Also, we should remark that 
%
%
\begin{equation} \label{eq:psiD4}
\psi(\ol{[\ell]})=\ol{[\ell]} \quad 
 \text{for every $\ell=0,\,1,\,2,\,3$}.
\end{equation}

%
\subsection{Niemeier lattice $\Ni(D_{4}^{6})$ and 
its automorphism $\sigma_{4}$ of order $3$.}
\label{subsec:Nie4}

Keep the notation and setting for the Niemeier lattice 
$L=\Ni(D_{4}^{6})$ with $Q=D_{4}^{6}$ the root lattice 
in \S\ref{subsec:Nie2}.
Let us define $\sigma_{4}:Q^{\ast} \rightarrow Q^{\ast}$ by: 
\begin{equation*}
\sigma_{4}
(\gamma_{1},\,\gamma_{2},\,\gamma_{3},\,\gamma_{4},\,\gamma_{5},\,\gamma_{6})=
(\psi(\gamma_{1}),\,\vp(\gamma_{2}),\,\vp^{-1}(\gamma_{3}),\,
 \gamma_{6},\,\vp^{-1}(\gamma_{4}),\,\vp(\gamma_{5})), 
\end{equation*}
where $\vp=\vp_{D_{4}}$ is the lattice automorphism of $D_{4}^{\ast}$ 
introduced in \S\ref{subsec:D4}. Then we deduce from \eqref{eq:vpD4} and 
\eqref{eq:psiD4} that $L \subset Q^{\ast}$ is stable under 
the action of $\sigma_{4}$; for example, we have, modulo $Q$, 
\begin{align*}
& [111111] \mapsto [123132] \equiv [032023]+[111111], \\
& [222222] \mapsto [231213] \equiv [023320]+[032023]+[222222], \\
& [002332] \mapsto [001221] \equiv [033202]+[032023],
\end{align*}
and so on. 
Therefore, $\sigma_{4}$ gives a lattice automorphism of $L$ of order $3$; 
since $\rank D_{4}^{\psi}=2$ and $\rank D_{4}^{\vp}=0$, it follows that 
$\rank L^{\sigma_{4}}=2+4=6$, and hence 
%
%
\begin{equation} \label{eq:rho4}
\rho=\frac{1}{18}(\dim \Fh_{(1)} + \dim \Fh_{(2)})=
\frac{1}{18}(24-6)=1. 
\end{equation}
Thus we can apply Theorem~\ref{thm:M} 
to $L=\Ni(D_{4}^{6})$ and the $\sigma_{4}$ above, and 
obtain a $C_{2}$-cofinite, holomorphic VOA $\ti{V}_{L}^{\sigma_{4}}$
of central charge $24$. 
We are ready to state the main result in this section. 
%
%
\begin{thm} \label{thm:main4}
Keep the notation and setting above. 
For $L=\Ni(D_{4}^{6})$ and the $\sigma_{4}$ above, 
the Lie algebra $(\ti{V}_{L}^{\sigma_{4}})_{1}$ is 
of type $A_{5,3}D_{4,3}A_{1,1}^{3}$. 
Therefore, $\ti{V}_{L}^{\sigma_{4}}$ corresponds to 
No.\,17 on Schellekens' list \cite[Table 1]{Sch}.  
\end{thm}

\begin{proof}
We will use the notation introduced in the proof of Theorem~\ref{thm:main2}. 
Recall that 
\begin{equation*}
(V_{L})_{1} \cong 
 \Fg^{(1)} \oplus \cdots \oplus \Fg^{(6)} \quad \text{with} \quad 
 \Fg^{(i)} \cong \Fg_{D_{4}} \quad \text{for every $1 \le i \le 6$}.
\end{equation*}
Then, $\sigma_{4}$ acts on $\Fg^{(1)}$ 
(resp. $\Fg^{(2)}$, $\Fg^{(3)}$) 
as the Lie algebra automorphism induced from $\psi$ 
(resp., $\vp$, $\vp^{-1}$). Also, we can easily check that 
the fixed point subalgebra of $\Fg^{(4)} \oplus \Fg^{(5)} \oplus \Fg^{(6)}$ 
under $\sigma_{4}$ is isomorphic to $\Fg_{D_{4}}$. Thus it follows that 
%
%
\begin{equation} \label{eq:s4-1}
(V_{L}^{\sigma_{4}})_{1} \cong 
  \Fg_{D_{4}}^{\psi} \oplus \Fg_{D_{4}}^{\vp} \oplus \Fg_{D_{4}}^{\vp^{-1}} 
  \oplus \Fg_{D_{4}}.
\end{equation}
In addition, by the same way as in the proof of Theorem~\ref{thm:main}, 
we deduce that $\Fg_{D_{4}}$, the last component of the right-hand side above, 
is a simple ideal of type $D_{4,3}$ of the whole Lie algebra $(\ti{V}_{L}^{\sigma_{4}})_{1}$. 
Since the dual Coxeter number of $D_{4}$ is equal to $6$, 
it follows immediately from \eqref{eq:level} that 
\begin{equation*}
\frac{6}{3}=
\frac{\dim (\ti{V}_{L}^{\sigma_{4}})_{1} -24}{24}, \quad \text{and hence} \quad
\dim (\ti{V}_{L}^{\sigma_{4}})_{1} = 72. 
\end{equation*}

Now, let us determine the Lie algebra structure of $\Fg_{D_{4}}^{\psi}$, 
the first component of the right-hand side of \eqref{eq:s4-1}. Observe that 
$\Fg_{D_{4}}^{\psi}$ is spanned by the following: 
\begin{align*}
& 
\bigl\{
 1 \otimes e^{\pm \beta_{j}^{(1)}} + 
 1 \otimes e^{\pm \psi(\beta_{j})^{(1)}} + 
 1 \otimes e^{\pm \psi^{2}(\beta_{j})^{(1)}} \mid j=1,\,2,\,3,\,4
\bigr\} \quad \text{(8 vectors)}; \\
&
\bigl\{
  h^{(1)}(-1) \otimes e^{0} \mid 
  h \in D_{4} \otimes_{\BZ} \BC, \, \psi(h)=h
\bigr\} \quad \text{(2-dimensional)};
\end{align*}
in particular, $\dim \Fg_{D_{4}}^{\psi}=10$. 
Let $\Fs$ be the Lie subalgebra of $\Fg_{D_{4}}^{\psi}$ generated by:
\begin{align*}
& 
\bigl\{
 1 \otimes e^{\pm \beta_{j}^{(1)}} + 
 1 \otimes e^{\pm \psi(\beta_{j})^{(1)}} + 
 1 \otimes e^{\pm \psi^{2}(\beta_{j})^{(1)}} \mid j=2,\,3,\,4
\bigr\} \quad \text{(6 vectors)}; \\
&
\bigl\{
  h^{(1)}(-1) \otimes e^{0} \mid 
  h \in D_{4} \otimes_{\BZ} \BC, \, \psi(h)=h
\bigr\} \quad \text{(2-dimensional)}.
\end{align*}
Then we deduce that $\Fs$ is an ideal of $\Fg_{D_{4}}^{\psi}$; 
indeed, it suffices to show that $\Fs$ is stable under the 
adjoint action of  $1 \otimes e^{\pm \beta_{1}^{(1)}} + 
 1 \otimes e^{\pm \psi(\beta_{1})^{(1)}} + 
 1 \otimes e^{\pm \psi^{2}(\beta_{1})^{(1)}}$. 
This follows from the facts that $\pair{h}{\beta_{1}}=0$ 
for all $h \in D_{4} \otimes_{\BZ} \BC$ such that $\psi(h)=h$, 
and that for any $0 \le r,\,s \le 2$ and $j=2,\,3,\,4$, 
\begin{equation*}
[1 \otimes e^{\pm \psi^{s}(\beta_{1})^{(1)}}, \, 
 1 \otimes e^{\pm \psi^{r}(\beta_{j})^{(1)}}]
\end{equation*}
is equal to either $0$ or a scalar multiple of 
$1 \otimes e^{\pm \psi^{t}(\beta_{k})^{(1)}}$ 
for some $0 \le t \le 2$ and $k=2,\,3,\,4$. 
Also, because $\pm \psi^{r}(\beta_{j})^{(1)} \in \Delta \setminus N$ 
for all $0 \le r \le 2$ and $j=2,\,3,\,4$, 
it follows from Lemma~\ref{lem:Ysh}\,(1) and (2) that 
$\Fs$ is also an ideal of the whole Lie algebra 
$(\ti{V}_{L}^{\sigma_{4}})_{1}$; in particular, 
$\Fs$ is a semisimple Lie algebra. Because $8 \le \dim \Fs \le 
10=\dim \Fg_{D_{4}}^{\psi}$, and because the dual Coxeter number of 
a simple component of $\Fs$ is an even integer by \eqref{eq:level}, 
we deduce that $\Fs \cong \Fg_{A_{1}}^{\oplus 3}$. 
Therefore the Lie algebra $\Fg_{D_{4}}^{\psi}$ is a reductive 
Lie algebra of the form: $\Fa \oplus \Fg_{A_{1}}^{\oplus 3}$, 
where $\Fa$ is a $1$-dimensional (abelian) Lie algebra. 

Set
\begin{equation*}
\Fg:=
 \underbrace{
 (\Fa \oplus \Fg_{D_{4}}^{\vp} \oplus \Fg_{D_{4}}^{\vp^{-1}})}_{%
  =:\Fg_{0}}
 \oplus 
 \underbrace{V(\sigma_{4})_{1}}_{=:\Fg_{1}} \oplus 
 \underbrace{V(\sigma_{4}^{2})_{1}}_{=:\Fg_{2}}.
\end{equation*}
By Remark~\ref{rem:Z3}, \eqref{eq:level}, and the argument above, 
we see that $\Fg$ is an ideal of $(\ti{V}_{L}^{\sigma_{4}})_{1}$ 
satisfying the following conditions: 

(i) $\Fg$ is a $(\BZ/3\BZ)$-graded, semisimple Lie algebra 
of dimension $35$; 

(ii) the dual Coxeter number of a simple component of $\Fg$ is 
an even integer. 

\noindent
Hence, $\Fg$ is of type
%
%
\begin{equation} \label{eq:a5}
A_{5} \quad \text{or} \quad C_{3}G_{2};
\end{equation}
in particular, $\Fg$ is of rank $5$. 
We will show that $\Fg$ is of type $A_{5}$. 

We see from the proof of Theorem~\ref{thm:main2} that 
both of $\Fg_{D_{4}}^{\vp}$ and $\Fg_{D_{4}}^{\vp^{-1}}$ 
are isomorphic to $\Fg_{A_{2}}$ 
(and hence $\Fg_{0} \cong \Fa \oplus \Fg_{A_{2}}^{\oplus 2}$). 
Let $\Ft_{1}$ and $\Ft_{2}$ be 
the Cartan subalgebras of $\Fg_{D_{4}}^{\vp} \cong \Fg_{A_{2}}$ 
and $\Fg_{D_{4}}^{\vp^{-1}} \cong \Fg_{A_{2}}$, respectively. 
Also, let $\Ft$ be the maximal subalgebra of $\Fg_{0}$, 
containing $\Ft_{1} \oplus \Ft_{2}$, such that 
the adjoint action of each element of $\Ft$ on $\Fg$ 
is semisimple. We see from \cite[Lemma~8.1\,b)]{Kac} that 
the centralizer $\Fh$ of $\Ft$ in $\Fg$ is 
a Cartan subalgebra of $\Fg$. It can be easily checked that 
$\Fh$ contains $\Fa \oplus \Ft_{1} \oplus \Ft_{2}$. 
Since $\Fg$ is of rank $5$, we have 
$\Fh=\Fa \oplus \Ft_{1} \oplus \Ft_{2}$. 
Thus we deduce that each root space of $\Fg$ (with respect to $\Fh$) 
is contained in exactly one of 
the graded spaces $\Fg_{0}$, $\Fg_{1}$, and $\Fg_{2}$. 
In particular, there exists a simple ideal $\Fu$ of $\Fg$ 
that contains $\Fg_{D_{4}}^{\vp} \cong \Fg_{A_{2}}$. 

Assume that the ratio of the squared length of a long root of $\Fu$ to 
the squared length of a root of $\Fg_{D_{4}}^{\vp} \cong \Fg_{A_{2}}$ 
is $r$ to $1$; note that $r \in \bigl\{1,\,2,\,3\bigr\}$. 
We deduce from the proof of Theorem~\ref{thm:main} that 
if we normalize the bilinear form $(\cdot\,,\,\cdot)_{\Fu}$ on $\Fu$ 
in such a way that the squared norm of a long root of 
$\Fu$ is equal to $2$, then for $x,\,y \in \Fg_{D_{4}}^{\vp}$ and 
$m,\,n \in \BZ$, 
\begin{equation*}
[x_{m},\,y_{n}]= 
(x_{0}y)_{m+n}+(3/r)m\delta_{m+n,\,0}(x,\,y)_{\Fu}\id_{\ti{V}_{L}^{\sigma_{4}}}
\quad \text{on $\ti{V}_{L}^{\sigma_{4}}$}.
\end{equation*}
Because the level of a simple component is a (positive) integer, 
the level of $\Fu$ is equal to either $1$ or $3$. 
If $\Fu$ were of type $C_{3}$ or $G_{2}$, then the level of $\Fu$ 
would be equal to $2$, which is a contradiction. Thus, by \eqref{eq:a5}, 
we conclude that $\Fg$ is of type $A_{5}$. 
This completes the proof of Theorem~\ref{thm:main4}. 
\end{proof}

%
\section{Construction of a holomorphic VOA (5).}
\label{sec:hvoa5}

%
\subsection{Automorphism of order $3$ on the root lattice $A_{5}$.}
\label{subsec:A5psi}

We use the notation introduced in \S\ref{subsec:An} (with $n=5$). 
Define a lattice automorphism $\psi=\psi_{A_{5}} \in \Aut(A_{5})$ by:
\begin{equation*}
(x_{0},\,x_{1},\,x_{2},\,x_{3},\,x_{4},\,x_{5})
 \mapsto
(x_{2},\,x_{0},\,x_{1},\,x_{5},\,x_{3},\,x_{4}).
\end{equation*}
Observe that $\psi$ acts on the set $\Delta(A_{5})$ of roots 
fixed-point-freely. Set
\begin{align*}
& \beta_{1} := (1,\,-1,\,0,\,0,\,0,\,0) \in \Delta(A_{5}), & 
& \beta_{2} := (0,\,0,\,0,\,1,\,0,\,-1) \in \Delta(A_{5}), & \\
& \beta_{3} := (1,\,0,\,0,\,-1,\,0,\,0) \in \Delta(A_{5}), & 
& \beta_{4} := (1,\,0,\,0,\,0,\,-1,\,0) \in \Delta(A_{5}), & \\
& \beta_{5} := (1,\,0,\,0,\,0,\,0,\,-1) \in \Delta(A_{5}). & 
\end{align*}
Then, $\bigl\{\pm \beta_{j} \mid 1 \le j \le 5\bigr\}$ 
is a complete set of representatives of $\psi$-orbits 
in $\Delta(A_{5})$. Also, it can be easily checked that 
%
%
\begin{equation} \label{eq:psiA5}
\psi(\ol{[\ell]})=\ol{[\ell]} \quad 
\text{for every $\ell=1,\,2,\,\dots,\,5$}.
\end{equation}

%
\subsection{Niemeier lattice $\Ni(A_{5}^{4}D_{4})$ and 
its automorphism $\sigma_{5}$ of order $3$.}
\label{subsec:Nie5}

The Niemeier lattice $L=\Ni(A_{5}^{4}D_{4})$ with $Q=A_{5}^{4}D_{4}$ 
the root lattice is, by definition (see \cite[Chapter 16, Table 16.1]{CS}), 
the sublattice of $Q^{\ast}=(A_{5}^{\ast})^{4}D_{4}^{\ast}$ 
generated by $Q$, $[33001]$, $[30302]$, $[30033]$, and 
$[2(024)0]=\bigl\{[20240],\,[22400],\,[24020]\bigr\}$, 
where $[a_{1} \cdots a_{5}]:=([a_{1}],\,\dots,\,[a_{5}]) \in 
Q^{\ast}=(A_{5}^{\ast})^{4}D_{4}^{\ast}$.
Let us define $\sigma_{5}:Q^{\ast} \rightarrow Q^{\ast}$ by: 
\begin{equation*}
\sigma_{5}(\gamma_{1},\,\gamma_{2},\,\gamma_{3},\,\gamma_{4},\,\gamma_{5})=
(\psi(\gamma_{1}),\,\gamma_{4},\,\gamma_{2},\,\gamma_{3},\,\vp(\gamma_{5})),
\end{equation*}
where $\vp=\vp_{D_{4}}$ is the lattice automorphism of $D_{4}^{\ast}$ 
introduced in \S\ref{subsec:D4}. Then we deduce from \eqref{eq:vpD4} and 
\eqref{eq:psiA5} that $L \subset Q^{\ast}$ is stable under 
the action of $\sigma_{5}$. 
Therefore, $\sigma_{5}$ gives a lattice automorphism of $L$ of order $3$; 
since $\rank A_{5}^{\psi}=1$ and $\rank D_{4}^{\vp}=0$, it follows that 
$\rank L^{\sigma_{5}}=1+5=6$, and hence 
%
%
\begin{equation} \label{eq:rho5}
\rho=\frac{1}{18}(\dim \Fh_{(1)} + \dim \Fh_{(2)})=
\frac{1}{18}(24-6)=1. 
\end{equation}
Thus we can apply Theorem~\ref{thm:M} 
to $L=\Ni(A_{5}^{4}D_{4})$ and the $\sigma_{5}$ above, and 
obtain a $C_{2}$-cofinite, holomorphic VOA $\ti{V}_{L}^{\sigma_{5}}$
of central charge $24$. 
We are ready to state the main result in this section. 
%
%
\begin{thm} \label{thm:main5}
Keep the notation and setting above. 
For $L=\Ni(A_{5}^{4}D_{4})$ and the $\sigma_{5}$ above, 
the Lie algebra $(\ti{V}_{L}^{\sigma_{5}})_{1}$ is 
of type $A_{5,3}D_{4,3}A_{1,1}^{3}$. 
Therefore, $\ti{V}_{L}^{\sigma_{5}}$ corresponds to 
No.\,17 on Schellekens' list \cite[Table 1]{Sch}.  
\end{thm}

\begin{proof}
We see that 
\begin{align*}
& (V_{L})_{1} \cong 
  \Fg^{(1)} \oplus \cdots \oplus \Fg^{(5)} \\
& \hspace*{20mm} 
  \text{with $\Fg^{(i)} \cong \Fg_{A_{5}}$ for every $1 \le i \le 4$, and 
  $\Fg^{(5)} \cong \Fg_{D_{4}}$}.
\end{align*}
Then, $\sigma_{5}$ acts on $\Fg^{(1)} \cong \Fg_{A_{5}}$ 
(resp., $\Fg^{(5)} \cong \Fg_{D_{4}}$) 
as the Lie algebra automorphism induced from $\psi=\psi_{A_{5}}$ 
(resp., $\vp=\vp_{D_{4}}$). Also, we can easily check that 
the fixed point subalgebra of $\Fg^{(2)} \oplus \Fg^{(3)} \oplus \Fg^{(4)}$ 
under $\sigma_{5}$ is isomorphic to $\Fg_{A_{5}}$. Thus it follows that 
%
%
\begin{equation} \label{eq:s5-1}
(V_{L}^{\sigma_{5}})_{1} \cong 
  \Fg_{A_{5}}^{\psi} \oplus \Fg_{A_{5}} \oplus \Fg_{D_{4}}^{\vp}.
\end{equation}
In addition, by the same way as in the proof of Theorem~\ref{thm:main}, 
we deduce that $\Fg_{A_{5}}$, the second component of the right-hand side above, 
is a simple ideal of type $A_{5,3}$ of 
the whole Lie algebra $(\ti{V}_{L}^{\sigma_{5}})_{1}$. 
Since the dual Coxeter number of $A_{5}$ is equal to $6$, 
it follows immediately from \eqref{eq:level} that 
\begin{equation*}
\frac{6}{3}=
\frac{\dim (\ti{V}_{L}^{\sigma_{5}})_{1} -24}{24}, \quad \text{and hence} \quad
\dim (\ti{V}_{L}^{\sigma_{5}})_{1} = 72. 
\end{equation*}

Now, let us determine the Lie algebra structure of 
the first component $\Fg_{A_{5}}^{\psi}$. Observe that 
$\Fg_{A_{5}}^{\psi}$ is spanned by the following: 
\begin{align*}
& 
\bigl\{
 1 \otimes e^{\pm \beta_{j}^{(1)}} + 
 1 \otimes e^{\pm \psi(\beta_{j})^{(1)}} + 
 1 \otimes e^{\pm \psi^{2}(\beta_{j})^{(1)}} \mid 1 \le j \le 5
\bigr\} \quad \text{(10 vectors)}; \\
&
\bigl\{
  h^{(1)}(-1) \otimes e^{0} \mid 
  h \in A_{5} \otimes_{\BZ} \BC, \, \psi(h)=h
\bigr\} \quad \text{(1-dimensional)};
\end{align*}
in particular, $\dim \Fg_{A_{5}}^{\psi}=11$. 
Let $\Fs$ be the Lie subalgebra of $\Fg_{A_{5}}^{\psi}$ generated by:
\begin{align*}
& 
\bigl\{
 1 \otimes e^{\pm \beta_{j}^{(1)}} + 
 1 \otimes e^{\pm \psi(\beta_{j})^{(1)}} + 
 1 \otimes e^{\pm \psi^{2}(\beta_{j})^{(1)}} \mid 3 \le j \le 5
\bigr\} \quad \text{(6 vectors)}; \\
&
\bigl\{
  h^{(1)}(-1) \otimes e^{0} \mid 
  h \in A_{5} \otimes_{\BZ} \BC, \, \psi(h)=h
\bigr\} \quad \text{(1-dimensional)}.
\end{align*}
Then we can show in exactly the same way as 
for $\Fs$ in the proof of Theorem~\ref{thm:main4} that 
$\Fs$ is an ideal of type $A_{1,1}^{3}$ of 
the whole Lie algebra $(\ti{V}_{L}^{\sigma_{4}})_{1}$. 
Therefore the Lie algebra $\Fg_{A_{5}}^{\psi}$ is a reductive 
Lie algebra of the form: $\Fa \oplus \Fg_{A_{1}}^{\oplus 3}$, 
where $\Fa$ is a $2$-dimensional (abelian) Lie algebra. 

Set
\begin{equation*}
\Fg:=
 \underbrace{
 (\Fa \oplus \Fg_{D_{4}}^{\vp})}_{%
  =:\Fg_{0}}
 \oplus 
 \underbrace{V(\sigma_{5})_{1}}_{=:\Fg_{1}} \oplus 
 \underbrace{V(\sigma_{5}^{2})_{1}}_{=:\Fg_{2}}.
\end{equation*}
By Remark~\ref{rem:Z3}, \eqref{eq:level}, and the argument above, 
we see that $\Fg$ is an ideal of $(\ti{V}_{L}^{\sigma_{5}})_{1}$ 
satisfying the following conditions: 

(i) $\Fg$ is a $(\BZ/3\BZ)$-graded, semisimple Lie algebra 
of dimension $28$; 

(ii) the dual Coxeter number of a simple component of $\Fg$ is 
an even integer. 

\noindent
Hence, $\Fg$ is of type
%
%
\begin{equation} \label{eq:d4}
D_{4} \quad \text{or} \quad G_{2}^{2};
\end{equation}
in particular, $\Fg$ is of rank $4$. 
We will show that $\Fg$ is of type $D_{4}$. 

Recall from the proof of Theorem~\ref{thm:main2} that 
$\Fg_{D_{4}}^{\vp}$ is isomorphic to 
the simple Lie algebra of type $A_{2}$. 
Let $\Ft$ be a Cartan subalgebra of $\Fg_{D_{4}}^{\vp} \cong \Fg_{A_{2}}$. 
By the same way as in Theorem~\ref{thm:main4}, 
we can show that $\Fh=\Fa \oplus \Ft$ is a Cartan subalgebra of $\Fg$. 
Thus each root space of $\Fg$ is contained in exactly one of 
the graded spaces $\Fg_{0}$, $\Fg_{1}$, and $\Fg_{2}$. 
In particular, there exists a simple ideal $\Fu$ of $\Fg$ 
that contains $\Fg_{D_{4}}^{\vp} \cong \Fg_{A_{2}}$. 
Then we can prove in exactly the same way as in Theorem~\ref{thm:main4} 
that the level of $\Fu$ is equal to either $1$ or $3$. 
If $\Fu$ were of type $G_{2}$ (see \eqref{eq:d4}), 
then the level of it would be equal to $2$, which is a contradiction. 
Thus we conclude that $\Fg$ is of type $D_{4}$. 
This completes the proof of Theorem~\ref{thm:main5}. 
\end{proof}

%
\appendix

\section{Appendix.}
\label{sec:apdx}

In \cite[\S5.2]{M}, Miyamoto proved that 
the VOA obtained by applying Theorem~\ref{thm:M} to 
the Niemeier lattice $L=\Ni(E_{6}^{4})$ with $Q=E_{6}^{4}$ 
the root lattice and a specified automorphism $\sigma_{6}$ 
of order $3$ corresponds to No.\,32 on Schellekens' list \cite{Sch}. 
In Appendix, after reviewing the definitions of $\Ni(E_{6}^{4})$ 
and $\sigma_{6}$, we give another proof for this fact, based on 
the Dong-Lepowsky construction of twisted modules 
(see \S\ref{subsec:twisted}) and the formula \eqref{eq:level}
due to Dong and Mason. 

%
\subsection{Niemeier lattice $\Ni(E_{6}^{4})$ and 
its automorphism of order $3$.}
\label{subsec:E6}

%
Let $\bigl\{\alpha_i\mid 1\le i\le 6\bigr\}$ be 
the simple roots for the root lattice 
$E_6=\bigoplus_{i=1}^{6}\BZ\alpha_{i}$; 
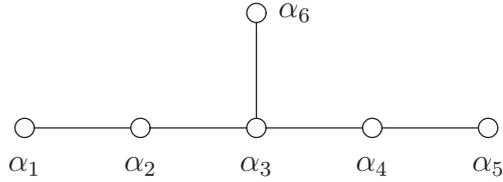
\begin{center}
%
%
\unitlength 0.1in
\begin{picture}( 29.0000,  8.0000)( -0.5000, -9.1500)
%
\special{pn 8}%
\special{ar 400 800 50 50  0.0000000 6.2831853}%
%
\special{pn 8}%
\special{ar 1000 800 50 50  0.0000000 6.2831853}%
%
\special{pn 8}%
\special{ar 1600 800 50 50  0.0000000 6.2831853}%
%
\special{pn 8}%
\special{ar 2200 800 50 50  0.0000000 6.2831853}%
%
\special{pn 8}%
\special{ar 2800 800 50 50  0.0000000 6.2831853}%
%
\special{pn 8}%
\special{ar 1600 200 50 50  0.0000000 6.2831853}%
%
\special{pn 8}%
\special{pa 450 800}%
\special{pa 950 800}%
\special{fp}%
%
\special{pn 8}%
\special{pa 1050 800}%
\special{pa 1550 800}%
\special{fp}%
%
\special{pn 8}%
\special{pa 1650 800}%
\special{pa 2150 800}%
\special{fp}%
%
\special{pn 8}%
\special{pa 2250 800}%
\special{pa 2750 800}%
\special{fp}%
%
\special{pn 8}%
\special{pa 1600 750}%
\special{pa 1600 250}%
\special{fp}%
\put(4.0000,-10.0000){\makebox(0,0){$\alpha_1$}}%
\put(10.0000,-10.0000){\makebox(0,0){$\alpha_2$}}%
\put(16.0000,-10.0000){\makebox(0,0){$\alpha_3$}}%
\put(22.0000,-10.0000){\makebox(0,0){$\alpha_4$}}%
\put(28.0000,-10.0000){\makebox(0,0){$\alpha_5$}}%
\put(18.0000,-2.0000){\makebox(0,0){$\alpha_6$}}%
\end{picture}%
\end{center}
\noindent
We set $$[0]:=0,\quad [1]:=\frac{1}{3}\left(\alpha_1-\alpha_2+\alpha_4-\alpha_5\right),\quad [2]:=\frac{1}{3}\left(-\alpha_1+\alpha_2-\alpha_4+\alpha_5\right),$$
and
\begin{equation*}
\ol{[0]} := [0]+E_{6}, \qquad
\ol{[1]} := [1]+E_{6}, \qquad
\ol{[2]} := [2]+E_{6}.
\end{equation*}
Then, $E_{6}^{\ast}/E_{6}=\bigl\{\ol{[\ell]} \mid \ell=0,\,1,\,2\bigr\}$, 
and the additive group $E_{6}^{\ast}/E_{6}$ is naturally 
isomorphic to $\BZ/3\BZ$. Denote by $\Delta(E_{6})$ 
the set of roots in $E_{6}$. 
Let $\vp=\vp_{E_{6}}$ denote the lattice automorphism of $E_6$ 
defined by: $\vp=r_1r_2r_4r_5r_6r_0$, 
where $r_i$ is the simple reflection with respect to 
the simple root $\alpha_i$ for $1 \le i \le 6$, and 
$r_{0}$ is the reflection with respect to the highest root of $E_{6}$; 
note that $\vp$ is of order $3$, and 
$\vp$ acts on $E_6 \otimes \BC$ fixed-point-freely. 
Also we see that $\vp(\ol{[\ell]})=\ol{[\ell]}$ for every $\ell=0,\,1,\,2$. 

The Niemeier lattice $L=\Ni(E_{6}^{4})$ with $Q=E_{6}^{4}$ the root lattice 
is, by definition (see \cite[Chapter 16, Table 16.1]{CS}), 
the sublattice of $Q^{\ast}=(E_{6}^{\ast})^{4}$ generated by 
$Q$ and $[1(012)]=\bigl\{[1012],\,[1120],\,[1201]\bigr\}$, 
where $[a_{1} \cdots a_{4}]:=([a_{1}],\,\dots,\,[a_{4}]) \in 
Q^{\ast}=(E_{6}^{\ast})^{4}$.
Let us define $\sigma_{6}:Q^{\ast} \rightarrow Q^{\ast}$ by: 
\begin{equation*}
(\gamma_{1},\,\gamma_{2},\gamma_{3},\gamma_{4}) \mapsto 
(\vp(\gamma_{1}),\,\gamma_{4},\gamma_{2},\gamma_{3}). 
\end{equation*}
Then we deduce that $L \subset Q^{\ast}$ 
is stable under the action of $\sigma_{6}$. 
Therefore, $\sigma_{6}$ gives a lattice automorphism of $L$ of order $3$. 
Since $\rank E_{6}^{\vp}=0$, it follows immediately that 
$\rank L^{\sigma_{6}}=6$, and hence 
%
%
\begin{equation} \label{eq:rho6}
\rho=\frac{1}{18}(\dim \Fh_{(1)} + \dim \Fh_{(2)})=
\frac{1}{18}(24-6)=1. 
\end{equation}
Thus we can apply Theorem~\ref{thm:M} 
to $L=\Ni(E_{6}^{4})$ and the $\sigma_{6}$ above, and 
obtain a $C_{2}$-cofinite, holomorphic VOA 
$\ti{V}_{L}^{\sigma_{6}}$ of central charge $24$. 
The following theorem is the main result in \cite[\S5.2]{M}. 
%
%
\begin{thm} \label{thm:main6}
Keep the notation and setting above. 
For $L=\Ni(E_{6}^{4})$ and the $\sigma_{6}$ above, 
the Lie algebra $(\ti{V}_{L}^{\sigma_{6}})_{1}$ is of type $E_{6,3}G_{2,1}^{3}$. 
Therefore, $\ti{V}_{L}^{\sigma_{6}}$ corresponds to 
No.\,32 on Schellekens' list \cite[Table 1]{Sch}.  
\end{thm}

\begin{proof} 
We deduce in exactly the same way as Lemma~\ref{lem:fixed} 
that $(V_{L}^{\sigma_{6}})_{1}=\Fg_{E_{6}}^{\vp} \oplus \Fg^{(234)}$, 
where $\Fg^{(234)}$ is isomorphic to the simple Lie algebra $\Fg_{E_{6}}$ of type $E_6$,
and $\Fg_{E_{6}}^{\vp}$ is the fixed point subalgebra of $\Fg_{E_{6}}$ 
under the Lie algebra automorphism induced from 
$\vp \in \Aut(E_{6})$; we deduce that $\Fg_{E_{6}}^{\vp}$ is 
isomorphic to the semisimple Lie algebra $\Fg_{A_{2}}^{\oplus 3}$.
%
%
%
Furthermore, we see by Lemma~\ref{lem:Ysh} that $\Fg^{(234)}$ is a simple ideal of 
the (whole) Lie algebra $(\ti{V}_{L}^{\sigma_{6}})_{1}$ 
whose level is equal to $3$. 
Since the dual Coxeter number of $E_{6}$ is equal to $12$, 
\begin{equation*}
\frac{12}{3}=\frac{\dim (\ti{V}_{L}^{\sigma_{6}})_{1} - 24}{24}, \quad 
\text{and hence} \quad \dim (\ti{V}_{L}^{\sigma_{6}})_{1} = 120.
\end{equation*}

Set
\begin{equation*}
\Fg:=\Fg_{0} \oplus \Fg_{1} \oplus \Fg_{2} \quad 
\text{with} \quad \Fg_{0}:=\Fg_{E_{6}}^{\vp}, \ 
\Fg_{1}:=V(\sigma_{6})_{1}, \ \Fg_{2}:=V(\sigma_{6}^{2})_{1}.
\end{equation*}
Then we see from the argument above, 
along with Proposition~\ref{prop:DM} and \eqref{eq:level}, 
that $\Fg$ is an ideal of $(\ti{V}_{L}^{\sigma_{6}})_{1}$ 
satisfying the following conditions: 

(i) $\Fg$ is a $(\BZ/3\BZ)$-graded, semisimple Lie algebra of dimension $42$; 

(ii) the dual Coxeter number of a simple component of $\Fg$ is contained in $4\BZ$. 

\noindent
Thus, $\Fg$ is isomorphic to the semisimple Lie algebra 
of type $G_2^3$ or $C_3^2$; in particular, 
the rank of $\Fg$ is equal to $6$. 
Hence the Cartan subalgebra of $\Fg_{E_{6}}^{\vp} \cong 
\Fg_{A_{2}}^{\oplus 3}$ is also a Cartan subalgebra of $\Fg$. 
Thus the root system of $\Fg$ contains the root system of $A_2^3$, 
which implies that $\Fg$ is of type $G_2^3$. 
Thus we have proved the proposition. 
\end{proof}


{\small
\setlength{\baselineskip}{13pt}
\renewcommand{\refname}{References}

}


\begin{thebibliography}{XXXXX}

\bibitem[CS]{CS}
J.H. Conway and N.J.A. Sloane, 
Sphere packing, lattices and groups, 
Third edition, Grundlehren der Mathematischen Wissenschaften, 
Vol.~290, Springer-Verlag, New York, 1999.

\bibitem[DGM]{DGM}
L.\ Dolan, P.\ Goddard and P.\ Montague, 
Conformal field theories, representations and lattice constructions, 
{\it Comm. Math. Phys.} {\bf 179} (1996), 61--120.

\bibitem[DL]{DL}
C. Dong and J. Lepowsky, 
The algebraic structure of relative twisted vertex operators, 
{\it J. Pure Appl. Algebra} {\bf 110} (1996), 259--295.

\bibitem[DLM1]{DLM98}
C. Dong, H. Li, and G. Mason, 
Twisted representations of vertex operator algebras, 
{\it Math. Ann.} {\bf 310} (1998), 571--600.

\bibitem[DLM2]{DLM00}
C. Dong, H. Li, and G. Mason, 
Modular-invariance of trace functions 
in orbifold theory and generalized Moonshine, 
{\it Comm. Math. Phys.} {\bf 214} (2000), 1--56.

\bibitem[DM1]{DM04}
C. Dong and G. Mason, 
Holomorphic vertex operator algebras of small central charge, 
{\it Pacific J. Math.} {\bf 213} (2004), 253--266.


\bibitem[DM2]{DM06}
C.\ Dong and G.\ Mason, Integrability of $C_2$-cofinite 
vertex operator algebras.  
{\it Int. Math. Res. Not.} (2006), Art. ID 80468, 15 pp.


\bibitem[FLM]{FLM}
I. Frenkel, J. Lepowsky and A. Meurman, 
Vertex operator algebras and the Monster, 
Pure and Appl. Math., Vol.~134, Academic Press, Boston, 1988.


\bibitem[ISS]{ISS}
M. Ishii, D. Sagaki, and H. Shimakura, 
in preparation.

\bibitem[K]{Kac}
V.G. Kac, Infinite-dimensional Lie algebras, 
Third edition, Cambridge University Press, Cambridge, 1990

\bibitem[La]{Lam}
C.H. Lam, On the constructions of holomorphic vertex operator algebras 
of central charge 24,
{\it Comm. Math. Phys.} {\bf 305} (2011), 153--198. 

\bibitem[LS1]{LS1}
C.H. Lam and H. Shimakura, 
Quadratic spaces and holomorphic framed vertex operator algebras of 
central charge 24, {\it Proc. Lond. Math. Soc. (3)}
{\bf 104} (2012), 540--576. 

\bibitem[LS2]{LS2}
C.H. Lam and H. Shimakura, 
Classification of holomorphic framed vertex operator algebras 
of central charge 24, to appear in Amer. J. Math., arXiv:1209.4677. 

\bibitem[LY]{LY}
C.H. Lam and H. Yamauchi, 
On the structure of framed vertex operator algebras and 
their pointwise frame stabilizers, 
{\it Comm. Math. Phys.} {\bf 277} (2008), 237--285. 

\bibitem[Le]{L}
J. Lepowsky, Calculus of twisted vertex operators, 
{\it Proc. Nat. Acad. Sci. U.S.A.} {\bf 82} (1985), 
8295--8299.

\bibitem[LL]{LL}
J. Lepowsky and H. Li, 
Introduction to vertex operator algebras and 
their representations, Progress in Mathematics Vol.~227, 
Birkh\"auser Boston, Inc., Boston, MA, 2004. 


\bibitem[Mi]{M}
M. Miyamoto, A $\BZ_3$-orbifold theory of 
lattice vertex operator algebra and 
$\BZ_3$-orbifold constructions, 
{\it in} ``Symmetries, Integrable Systems and Representations'', 
Springer Proceedings in Mathematics and Statistics Vol.~40, 
pp.319--344, Springer-Verlag, London, 2013. 

\bibitem[S]{Sch}
A.N. Schellekens, Meromorphic $c=24$ conformal field theories, 
{\it Comm. Math. Phys.} {\bf 153} (1993), 159--185.

\bibitem[Z]{Z}
Y. Zhu, Modular invariance of characters of vertex operator algebras, 
{\it J. Amer. Math. Soc.} {\bf 9} (1996), 237--302.

\end{thebibliography}
\end{document}